\documentclass[leqno,twoside, 10pt]{amsart}
\usepackage[TS1,T1]{fontenc}
\usepackage{amsfonts,amssymb,amsmath,amsthm,enumerate,latexsym}

\usepackage{amssymb, latexsym, esint}
\usepackage{color}

\newcommand{\grad}{\nabla}
\newcommand{\dint}{\dyle\int}

\newcommand{\p}{\partial}

\newcommand{\re}{{I\!\!R}}
\newcommand{\ren}{\re^N}

\newcommand{\dyle}{\displaystyle}
\newcommand{\ene}{{I\!\!N}}

\newcommand{\io}{\int\limits_\O}

\renewcommand{\a }{\alpha }
\renewcommand{\b }{\beta }
\renewcommand{\d }{\delta }
\newcommand{\D }{\Delta }
\newcommand{\e }{\varepsilon }

\renewcommand{\l }{\lambda }

\newcommand{\n }{\nabla }

\newcommand{\s }{\sigma }

\renewcommand{\O }{\Omega }

\newcommand{\inn}{\mbox{ in }}

\newtheorem{Theorem}{Theorem}[section]

\theoremstyle{plain} \numberwithin{equation}{section} \numberwithin{figure}{section}

\newtheorem{theorem}{Theorem}[section]
\newtheorem{lemma}[theorem]{Lemma}
\newtheorem{proposition}[theorem]{Proposition}
\newtheorem{corollary}[theorem]{Corollary}
\newtheorem{definition}[theorem]{Definition}
  \usepackage{epsfig} 
\usepackage{graphicx}  \usepackage{epstopdf}
\theoremstyle{definition}
\newtheorem{remark}[theorem]{Remark}

\begin{document}
\title[Quasilinear equations with fractional diffusion] {A note on quasilinear equations with fractional diffusion}
\author[B. Abdellaoui, P. Ochoa, I. Peral]{Boumediene Abdellaoui, Pablo Ochoa, Ireneo Peral}
\address{\hbox{\parbox{5.7in}{\medskip\noindent {B. Abdellaoui, Laboratoire d'Analyse Nonlin\'eaire et Math\'ematiques
Appliqu\'ees. \hfill \break\indent D\'epartement de Math\'ematiques, Universit\'e Abou Bakr Belka\"{\i}d, Tlemcen, \hfill\break\indent Tlemcen 13000, Algeria.}}}}
\email{boumediene.abdellaoui@inv.uam.es}
\address{P. Ochoa, Universidad Nacional de Cuyo-CONICET, 5500 Mendoza, Argentina}
\email{ochopablo@gmail.com}
\address{I. Peral, Universidad Aut\'onoma de Madrid,  28049 Madrid, Spain}
\email{ireneo.peral@uam.es}



\thanks{ The first and the third authors are partially supported by Project MTM2016-80474-P, MINECO, Spain. The  first author also is partially supported by DGRSDT, Algeria and an Erasmus grant from Autonoma University of Madrid. The second author is partially supported by CONICET and grant PICT 2015-1701 AGENCIA.}

\date{\today}

\maketitle

\rightline{\textit{To Alberto Farina in his 50th birthday with our friendship}}

\begin{abstract}

In this paper, we study the existence of distributional solutions of the following non-local elliptic problem
\begin{eqnarray*}
 \left\lbrace
  \begin{array}{l}
    (-\Delta)^{s}u + |\nabla u|^{p} =f \quad\text{ in } \Omega\\
    \qquad \qquad \qquad \,\,\, u=0  \,\,\,\,\,\,\,\text{ in } \mathbb{R}^{N}\setminus \Omega, \quad s \in (1/2, 1). \\
  \end{array}
  \right.
\end{eqnarray*}
 We are interested in the relation between the regularity of the source term $f$, and the regularity of the corresponding solution. If $1 <  p<2s$, that is the natural
 growth, we are able to show the existence for all $f\in L^1(\O)$.

 In the subcritical case, that is, for $1 < p < p_{*}:=N/(N-2s+1)$, we show that solutions are $\mathcal{C}^{1, \alpha}$ for $f \in L^{m}$, with $m$ large enough. In
 the general case, we achieve the same result under a condition on the size of the source. As an application, we may show that for regular sources,
 distributional solutions are viscosity solutions, and conversely.
\end{abstract}

\section{Introduction}

Throughout this article, we shall consider the following  Dirichlet integro-differential problem
\begin{eqnarray}\label{boundary problem}
 \left\lbrace
  \begin{array}{l}
    (-\Delta)^{s}u + |\nabla u|^{p} =f \quad \,\, \,\text{in } \Omega\\
    \qquad \qquad \qquad \,\,\, u=0  \,\,\,\,\,\,\,\,\text{ in } \mathbb{R}^{N}\setminus \Omega, \\
  \end{array}
  \right.
\end{eqnarray}for $s \in (1/2, 1)$, $\Omega \subset \mathbb{R}^{N}$, $p>1$ and $f$ a non-negative measurable function. When the nonlinear term appears in the
righthand side the model \eqref{boundary problem} may be seen as a Kardar-Parisi-Zhang stationary problem driving by  fractional diffusion (see \cite{KPZ} for the
model in the local setting and \cite{AP} in the nonlocal case). The problem  with the nonlinear term in the left hand side is the stationary counterpart of a
Hamilton-Jacobi equation with a viscosity term, the principal nonlocal operator. See \cite{Sil} and the references therein.

The fractional Laplacian operator $(-\Delta)^{s}$, and more general pseudo-differential operators, have been a classic topic in Harmonic Analysis and PDEs.
Moreover, these is a  renovated interest in these kind of operators. Non-local operators arise naturally in continuum mechanics, image processing, crystal
dislocation, phase transition phenomena, population dynamics, optimal control and theory of games as pointed out in \cite{BC}, \cite{BV}, \cite{C}, \cite{CVa},
\cite{GO} and the references therein. For instance, the fractional heat equation may appear in probabilistic random-walk procedures and, in turn, the stationary
case may do so in pay-off models (see \cite{BV} and the references therein). In the works \cite{MK} and \cite{MK1} the description of anomalous diffusion via
fractional dynamics is investigated and various fractional partial differential equations are derived from L\'evy random walk models, extending Brownian walk
models in a natural way. Fractional operators are also involved in financial mathematics, since L\`{e}vy processes with jumps revealed as more appropriate models
of stock pricing. The \textit{bounday condition}
$$u=0 \,\,\textnormal{ in } \,\mathbb{R}^{N}\setminus \Omega$$which is given in the whole complement may be interpreted  from the stochastic point of view as the
fact that a L\`{e}vy process can exit the domain $\Omega$ for the first time jumping to any point in its complement.

Regarding the integro-differential problem that we discuss in the present manuscript,  the main results of our research may be summarized as follows

\begin{itemize}
  \item In the sub-critical scenario $1<p < p_{*}:=\frac{N}{N-2s+1}$, there is a unique non-negative distributional solution $u \in W^{1, q}_0(\Omega)$ of  \eqref{boundary problem} for
      any $1\leq q < p_{*}$.
  \item Now, for $x\in \O$, setting $\d(x)=\text{dist}(x,\p\O)=\textnormal{dist}(x, \Omega^{c})$ (since $\O$ is a bounded regular domain), then if $1<p<p_{*}$, with similar arguments to those in \cite{AP} and \cite{CV}, we have
 \begin{itemize}
 \item If $m < \frac{N}{2s-1}$, then $|\nabla u|\d^{1-s} \in L^{q}(\Omega)$ for all $1\leq q < \frac{mN}{N-m(2s-1)}$.
 \item If $m= \frac{N}{2s-1}$, then $|\nabla u|\d^{1-s} \in L^{q}(\Omega)$ for all $ 1\leq q<\infty$.
 \item If $m > \frac{N}{2s-1}$, then $|\nabla u| \in \mathcal{C}^{\alpha}(\Omega)$ for some $\alpha \in (0,1)$.
\end{itemize}
In the interval $1< p < p_{*}$ the result lies on the estimates for the Green function by Bogdan and Jakubowski in \cite{BJ}. \item For any $1 < p < \infty$,
$u$ is $\mathcal{C}^{1, \alpha}$ provided the source is sufficiently small.

\item Any  solution $u \in \mathcal{C}^{1, \alpha}(\Omega)$ with H\"{o}lder continuous source is a viscosity solution, and conversely.
\end{itemize}

Notice that in the local case $s=1$, the main existing results can be summarized into two points: If $p\le 2$, then the existence of solution is obtained for all
$f\in L^1(\O)$ using approximation arguments and suitable test function, see \cite{BGO} and the references therein. However the truncating arguments are not
applicable for $p>2$ including for $L^\infty$ data. In the case of Lipschitz data, the author in \cite{PPL} was able to get the existence and the uniqueness of a
regular solution for all $p$. However this last argument is not applicable for $L^m$ data including for $p$ close to two.

For the non local case, the first existence result was obtained in \cite{CV}.  Indeed, they consider the problem
 \begin{eqnarray}\label{CV1}
 \left\lbrace
  \begin{array}{l}
    (-\Delta)^{s}u + \epsilon g\left(|\nabla u| \right)=\nu  \quad\,\,\text{ in } \Omega\\
    \,\,\quad \qquad \qquad \qquad \,\,\, u=0  \,\,\,\,\,\,\,\,\text{ in } \mathbb{R}^{N}\setminus \Omega, \quad s \in (1/2, 1), \,\\
  \end{array}
  \right.
\end{eqnarray}with $ \epsilon \in \left\lbrace -1, 1 \right\rbrace$, for a continuous and non-negative function $g$ satisfying $g(0)=0$ and  a non-negative Radon
measure $\nu$ so that $\dyle\int_\Omega \delta^{\beta}d\nu < \infty$ with $\beta \in [0, 2s-1)$.

In \cite[Thm. 1.1]{CV}, they show
that for $\epsilon=1$ and  under the integrability  assumption
 $$\int_{1}^{\infty}g(s)s^{-1-p^{*}}ds < \infty,$$problem \eqref{CV1} admits a non-negative distributional solution $u \in W^{1, q}_0(\Omega)$, for all $1 \leq q <
 p_{*, \beta}$ where
 $$p_{*, \beta}:=\frac{N}{N-2s+1+\beta}.$$In particular, this result implies that the Dirichlet problem \eqref{boundary problem}  admits a solution $u$ in $W^{1,
 q}_0(\Omega)$ for all $q \in [1, p_{*})$ and for $1<p < p_{*}$. Moreover, for $g$ H\"{o}lder continuous and bounded in $\mathbb{R}$, solutions to \eqref{CV1}
 becomes strong for a H\"{o}lder continuous source.

The regularity of solutions to \eqref{boundary problem} is strongly related to the corresponding issue for problems
 \begin{eqnarray}\label{gener1-corr}
 \left\lbrace
  \begin{array}{l}
    (-\Delta)^{s}v =f  \quad\,\,\text{ in } \Omega\\
  \,\,  \qquad \,\,\, v=0  \,\,\,\,\,\,\,\,\text{ in } \mathbb{R}^{N}\setminus \Omega, \\
  \end{array}
  \right.
\end{eqnarray}
As a by-product of the results in \cite{AP}, \cite{CV} and \cite{CV1}, we have the following result which will be largely used throughout our paper.

\begin{theorem}\label{th AP}
Suppose that $f\in L^m(\O)$ with $m\ge 1$ and define $v$ to be the unique solution to problem \eqref{gener1-corr} with $s>\frac12$. Then for all $1\leq p<\frac{m
N}{N-m(2s-1)}$, there exists a positive constant $C\equiv \hat{C}(\O, N, s,p)$ such that
\begin{equation}\label{dd11-corr}
\bigg\||\n v| \d^{1-s}\bigg\|_{L^p(\Omega)}\le\hat{ C}||f||_{L^m(\O)}.
\end{equation}
Moreover,
\begin{enumerate}
\item If $m=\frac{N}{2s-1}$, then $|\n v|\d^{1-s}\in L^p(\Omega)$ for all $1\leq p<\infty$. \item If $m>\frac{N}{2s-1}$, then $v\in \mathcal{C}^{1, \sigma}(\Omega)$
    for some $\sigma\in (0,1)$, and
$$
\bigg\| |\n v|\d^{1-s}\bigg\|_{L^\infty(\O)}\le C||f||_{L^m(\O)}.
$$
\end{enumerate}
\end{theorem}

In the case where $f\in L^1(\O)\cap L^m_{loc}(\O)$ where $m>1$, then as it was proved in \cite{AP}, the above regularity results hold locally in $\O$. More
precisely we have
\begin{proposition}\label{key2-locc}
Assume that $f\in L^1(\O)\cap L^m_{loc}(\O)$ with $m>1$. Let $v$ the unique solution to problem \eqref{gener1-corr}. Suppose that $m<\frac{N}{2s-1}$, then for any
$\O_1\subset\subset \O'_1\subset \subset \O$ and for all $1\leq p\le \frac{m N}{N-m(2s-1)}$, there exists $\tilde{C}:=\tilde{C}(\O,\O_1,\O'_1,N,s,p)$ such that
\begin{equation}\label{dd11loc}
||\n v||_{L^p(\Omega_1)}\le \tilde{C}(||f||_{L^1(\O)}+||f||_{L^m(\O'_1)}).
\end{equation}
Moreover,
\begin{enumerate}
\item If $m=\frac{N}{2s-1}$, then $|\n v|\in L^p_{loc}(\Omega)$ for all $1 \leq p<\infty$. \item If $m>\frac{N}{2s-1}$, then $v\in \mathcal{C}^{1, \sigma}(\Omega)$
    for some $\sigma\in (0,1)$.
\end{enumerate}
\end{proposition}

 As a consequence we conclude that, if $f\in L^m(\O)$ with $m>1$, then
\begin{enumerate}
\item If $m\ge\frac{N}{2s-1}$, then $\dyle\io |\n v|^a dx<\infty$ for all $a<\frac{1}{1-s}$. \item If $1<m<\frac{N+2s}{2s-1}$, then $\dyle \io |\n v|^a
    dx<\infty$ for all $a<\check{P}:=\frac{m N}{N(m(1-s)+1)-m(2s-1)}$.
\end{enumerate}
\begin{remark}
It is clear that $a<a_0=\frac{1}{1-s}$ is optimal. Before proving the optimality of $a_0$, let us recall the next Hardy inequality that will be used systematically in what follows.
\begin{proposition}\label{hardydy}(Hardy inequality)
Assume that $\O$ is a bounded regular domain of $\ren$ and $1<p<N$. Then there exists a positive constant $C(\O)$ such that for all $\phi\in W^{1,p}_0(\O)$, we have
\begin{equation}\label{hardydye}
C(\O)\dint_\Omega \frac{|\phi|^p}{\d^p}dx\le \dint_\Omega |\nabla \phi|^pdx<+\infty.
\end{equation}
\end{proposition}
We prove now the optimality of $a_0$. We argue by contradiction. Assume that, for $0\lneqq f\in L^\infty(\O)$, there exists a solution $v$ to
\eqref{gener1-corr} such that $v\in W^{1,p}_0(\O)$ with $p>\frac{1}{1-s}$.

By using the classical Hardy inequality we obtain that
$$\dint_\Omega \frac{v^p}{\d^p}dx\le \dint_\Omega |\nabla v|^pdx<+\infty.$$
By the results in \cite{RS} the solution behaves  as $v\backsimeq \d^s$, therefore, as a consequence, $\dfrac{1}{\d^{p(1-s)}}\in L^1(\O)$, that is,
$p<\frac{1}{1-s}$, a contradiction.

Hence, the bound for the exponent of the gradient seems to be natural if we impose that the solution lies in the Sobolev space $W^{1,p}_0(\O)$ for the problem
with reaction gradient term.

In the case of absorption gradient term, this affirmation seems to be difficult to prove, however, in Theorem \ref{non}, we will show that the non existence
result holds, at least, for large values of $p$ and for all bounded non negative data.

In the case of gradient reaction term and for $2s\le p<\dfrac{s}{1-s}$, the authors in \cite{AP} proved the existence of a solution $u$ with $|\n u|\in
L^p_{loc}(\O)$ using a fixed point argument. In the present paper we will use the same approach to get the existence of a solution for $p\ge 2s$. However, in
addition to the regularity condition of $f$, smallness condition on the source term $||f||_{L^m(\O)}$ is also needed.
\end{remark}

The paper is organized as follows. In Section \ref{into}, we introduce the functional setting and we precise the notion of solutions that we will use throughout
this work as the weak sense and the viscosity sense. We give also some useful estimates for weak solutions and the general comparison principle. A non existence
result is proved using suitable estimates on the Green function for the fractional Laplacian with drift term.

The existence of a solution is proved in Section \ref{exis}. In the Subsection \ref{natural}  we treat the case of natural growth behavior in the gradient term, namely the case $1 < p<2s$. In this case existence of a solution is obtained for all $L^1$ datum. As a complement of the result proved in \cite{CV}, we prove that if $p>p_*$, the existence of a solution for general measure data $\nu$ is not true and additional hypotheses on $\nu$ related to a fractional capacity are needed.

Problem with a linear zero order reaction term is also analyzed. In such a case we are able to show existence for data in $L^1$ and then
a breaking of resonance occurs under natural hypotheses on the zero order term and $p$.

Some additional regularity results are obtained in the subcritical case $1 <p<p_*$.

 The  general case,  $p\ge 2s$, is treated in Subsection \ref{general}. Here and since we will use fixed point theorem, we need to impose some additional condition on the regularity and the size of $f$. The existence result is obtained in a suitable weighted Sobolev space under additional hypotheses on $p$. The above existence result holds trivially for the case $s=1$ and then  can be seen as an extension of the existence result obtained in \cite{PPL} in the framework of $L^m$ datum.

The analysis of the viscosity solution is done is Section \ref{equiv} where it is also proved that weak solution is a viscosity solution and viceversa if the data $f$ is sufficiently regular and $s$ is close to 1.

Some related open problems are given in the last section.

\subsection{Basic notation}In what follows, $\Omega$ will denote a bounded, open and $\mathcal{C}^{2}$ domain in $\mathbb{R}^{N}$ with bounded boundary, $N \geq
1$. We introduce some functional-space  notation. By $USC(\Omega)$, $LSC(\Omega)$ and $\mathcal{C}(\Omega)$, we denote the spaces of upper semi-continuous, lower
semi continuous and continuous real-valued functions in $\Omega$, respectively. Moreover, the space $\mathcal{C}^{k}(\Omega)$, $k \geq 1$, is defined as the set
of functions which derivatives of orders $\leq k$ are continuous in $\Omega$. Also, the H\"{o}lder space $\mathcal{C}^{k, \alpha}(\Omega)$ is the set of
$\mathcal{C}^{k}(\Omega)$ whose $k-$th order partial derivatives are  locally H\"{o}lder continuous with exponent $\alpha$ in $\Omega$.

For $\sigma \in \mathbb{R}$, we define the truncation operator as follows
$$T_k(\sigma):=\max(-k, \min(k, \sigma)).$$Finally, for any $u$, we denote by
$$u_+= \max\left\lbrace 0, u \right\rbrace \quad \textnormal{ and } \quad u_{-}=\max\left\lbrace 0, -u \right\rbrace.$$

\section{Preliminaries and technical tools.}\label{into}

In order to introduce the notion of distributional solutions, we give some definitions. For $s \in (\frac12, 1)$ and $u\in\mathcal{S}(\mathbb{R}^{N})$,  the fractional
Laplacian $(-\Delta)^{s}$ is given by
$$(-\Delta)^{s}u(x):=\lim_{\epsilon \to 0}(-\Delta)^{s}_{\epsilon}u(x)$$
where
$$(-\Delta)^{s}_\epsilon u(x):= \int_{\mathbb{R}^{N}}\frac{u(x)-u(y)}{|x-y|^{N+2s}}\chi_{\epsilon}(|x-y|)dy$$with:

\begin{equation*}
\chi_t(|x|) := \left\lbrace
  \begin{array}{l}
  0,  \quad |x|< t\\
    1,  \quad |x|\geq t. \\
  \end{array}
  \right.
\end{equation*}
For larger class of functions the fractional Laplacian can be defined by density. See \cite{dine} or \cite{sil} for instance.
\begin{definition}\label{test} We say that a function $\phi \in \mathcal{C}(\mathbb{R}^{N})$ belongs to $\mathbb{X}_s(\Omega)$ if and only if the following holds
\begin{itemize}
\item supp$(\phi) \subset \overline{\Omega}$. \item The fractional Laplacian  $(-\Delta)^{s}\phi(x)$ exists  for all  $x \in \Omega$ and there is $C>0$ so
    that $|(-\Delta)^{s}\phi(x)|\leq C$. \item There is $\varphi \in L(\Omega, \delta^{s}dx)$ and $\epsilon_0 >0$ so that
$$|(-\Delta)^{s}_\epsilon \phi(x)|\leq \varphi(x),$$a. e. in $\Omega$ and for all $\epsilon \in (0, \epsilon_0)$.
\end{itemize}
\end{definition}

Before staring the sense for which solutions are defined, let us recall the definition of the fractional Sobolev space and some of its properties.

Assume that $s\in (0,1)$ and $p>1$. Let $\O\subset \ren$, then the fractional Sobolev Space $W^{s,p}(\Omega)$ is defined by
$$
W^{s,p}(\Omega)\equiv \Big\{ \phi\in L^p(\O):\iint_{\O\times \O}|\phi(x)-\phi(y)|^pd\nu<+\infty\Big\},
$$
where $d\nu=\dyle\frac{dxdy}{|x-y|^{N+ps}}$.

Notice that $W^{s,p}(\O)$ is a Banach Space endowed with the norm
$$
\|\phi\|_{W^{s,p}(\O)}= \Big(\dint_{\O}|\phi(x)|^pdx\Big)^{\frac 1p} +\Big(\iint_{\O\times\O}|\phi(x)-\phi(y)|^pd\nu\Big)^{\frac 1p}.
$$
The space $W^{s,p}_{0} (\O)$ is defined as the completion of $\mathcal{C}^\infty_0(\O)$ with respect to the previous norm.

If $\O$ is a bounded regular domain, we can endow $W^{s,p}_{0}(\O)$ with the equivalent norm
$$
||\phi||_{W^{s,p}_{0}(\O)}= \Big(\iint_{\O\times \O}|\phi(x)-\phi(y)|^pd\nu\Big)^{\frac 1p}.
$$
Notice that if $ps<N$, then we have the next Sobolev inequality, for all $v\in C_{0}^{\infty}(\ren)$,
$$
\iint_{\re^{2N}} \dfrac{|v(x)-v(y)|^{p}}{|x-y|^{N+ps}}\,dxdy\geq S \Big(\dint_{\mathbb{R}^{N}}|v(x)|^{p_{s}^{*}}dx\Big)^{\frac{p}{p^{*}_{s}}},
$$
where $p^{*}_{s}= \dfrac{pN}{N-ps}$ and $S\equiv S(N,s,p)$.

In the following definition, we introduce the class of distributional solutions.

Assume that $\nu$ is a bounded Radon measure and consider the problem
\begin{equation}\label{eq:def}
\begin{cases}
(-\Delta )^s v=\nu &\hbox{   in   }  \Omega,\\ v=0   &\hbox{   in   } \mathbb{R}^N\setminus\Omega,
\end{cases}
\end{equation}
Let us begin by precising the sense in which solutions are defined for general class of data.
\begin{definition}\label{def1}
We say that $u$ is a weak solution to problem \eqref{eq:def} if $u\in L^1(\O)$, and for all $\phi\in \mathbb{X}_s$, we have
$$
\io u(-\Delta )^s\phi dx =\io \phi d\nu,
$$
where $\mathbb{X}_s$ is  given  in  Definition \ref{test}.
\end{definition}

As a consequence of the properties of the Green function, the authors in \cite{CV1} obtain the following regularity result.
\begin{Theorem}\label{key}
Suppose that $s\in (\frac 12, 1)$ and let $\nu \in \mathfrak{M}(\O)$, be a Radon measure such that
$$\int_\Omega \delta^{\beta}d\nu < \infty, \qquad \delta(x):= \textnormal{dist}(x, \Omega^{c}),$$ with $\beta \in [0, 2s-1)$.
Then the problem \eqref{eq:def} has a unique weak solution $u$ in the sense of Definition \ref{def1}  such that $u \in W^{1, q}_0(\Omega)$, for all $1\leq q <
p^{*}_{\beta}$ where $p_{*, \beta}:=\frac{N}{N-2s+1+\beta}.$ Moreover
\begin{equation}\label{dd}
||u||_{W^{1,q}_0(\Omega)}\le C(N,q,\O)\io \d^\beta d\nu.
\end{equation}
For $\nu \in L^1(\O)$, setting\; $T: L^1(\O)\to  W^{1,\theta}_0(\Omega)$, with $T(f)=u$, then $T$ is a compact operator.
\end{Theorem}

 Related to $T_k(u)$ and for $s>\frac 12$, we have the next regularity result obtained in \cite{AP}.
\begin{Theorem}\label{key01}
Assume that $f\in L^1(\Omega)$ and define $u$ to be the unique weak solution to problem \eqref{eq:def}, then $T_k(u)\in W^{1,\a}_0(\Omega)\cap H^s_0(\Omega)$  for
any $\a<2s$, moreover $$ \io |\n T_k(u)|^{\a}\, dx\le Ck^{\a-1}||f||_{L^1(\O)}.
$$
\end{Theorem}

We recall also the next comparison principle proved in \cite{AP}
\begin{Theorem}\label{compa2}(Comparison Principle).
Let $g\in L^1(\O)$ and suppose that $w_1, w_2\in W^{1,q}_0(\O)$ for all $1 \leq q<\frac{N}{N-2s+1}$ are such that $(-\D)^s w_1, (-\D)^s w_2\in L^1(\O)$ with
$$\begin{array}{ll}
\begin{cases}(-\Delta)^s w_1 \leq H_1(x,w_1,\n w_1)  +g \hbox{ in }\Omega, \\
w_1 \leq 0 \inn \mathbb{R}^{N}\setminus\Omega,
\end{cases}

&

\begin{cases}
(-\Delta)^s w_2\geq H_1(x,w_2,\n w_2)+g\hbox{ in }\O, \\ w_2\leq 0\inn \mathbb{R}^{N}\setminus\Omega,
\end{cases}
\end{array}$$
\

where $H: \O\times \re\times \ren\to \re$ is a Carath\'eodoty function satisfying
\begin{enumerate}
\item $H_1(x,w_1,\n w_1), H_1(x,w_2,\n w_2)\in L^1(\O)$,
\item for a.e. $x\in \O$, we have
$$ H_1(x,w_1,\n w_1)-H_1(x,w_2,\n w_2)=\langle B(x,w_1,w_2, \nabla w_1, \nabla w_2), \n (w_1-w_2)\rangle +
    f(x,w_1,w_2)$$
with $B\in (L^{a}(\O))^N$, $a>\frac{N}{2s-1}$ and $f\in L^1(\O)$ with $f\le 0$ a.e.  in $\O$.
\end{enumerate}
Then $w_1\le w_2$ in $\O$.
\end{Theorem}

Recall that we are considering problem \eqref{boundary problem}, then we have the next definition.

\begin{definition}A function $u \in L^{1}(\Omega)$, with $|\nabla u|^{p}\in L^{1}_{loc}(\Omega)$, is a distributional solution to problem \eqref{boundary problem}
if for any $\phi \in \mathbb{X}_s(\Omega)$, there holds
$$\int_\Omega u(-\Delta)^{s}\phi + \int_\Omega\phi |\nabla u|^{p} = \int_\Omega f\phi,$$and $u=0$ in $\mathbb{R}^{N}\setminus \Omega$.
\end{definition}

We denote by $G_s$ the Green kernel of $(-\Delta)^{s}$ in $\Omega$ and  by  $\mathbb{G}_s[\cdot]$ the associated Green operator defined by
$$\mathbb{G}_s[f](x):=\int_\Omega G_s(x, y)df(y).$$
See \cite{BJ} for the estimates of the Green  function.

\begin{definition}\label{strong}A function $u: \Omega \to \mathbb{R}$ is a strong solution to the equation
$$  (-\Delta)^{s}w + |\nabla w|^{p} =f$$ in  $\Omega$ if $u \in \mathcal{C}^{2s+ \alpha}(\Omega)$, for some $\alpha >  0$ and $$  (-\Delta)^{s}u(x) + |\nabla
u(x)|^{p} =f(x)$$ for every $x$ in $\Omega$.

\end{definition}

The other class of solutions that we shall consider is the class of viscosity solutions.  Unlike the distributional scenario, the notion of viscosity solutions
requires the punctual evaluation of the equation using appropriate test functions that touch the solution from above or below.


 \begin{definition}\label{defi viscosity}
An upper semicontinuous function $u: \mathbb{R}^{N} \to \mathbb{R}$ is a viscosity subsolution to \eqref{boundary problem} in $\Omega$, if $u \in
L_{loc}(\mathbb{R}^{N})$, and for any open set $U \subset \Omega$, any $x_0 \in U$ and any $\phi \in \mathcal{C}^{2}(U)$ such that $u(x_0)=\phi(x_0)$ and $\phi
\geq u$ in $U$, if we let

\begin{equation}\label{function v}
 v(x):= \left\lbrace
  \begin{array}{l}
  \phi(x)  \text{ in } U\\
    u(x)  \text{ outside  }U, \\
  \end{array}
  \right.
\end{equation}we have
$$(-\Delta)^{s}\phi(x_0) + |\nabla \phi(x_0)|^{p} \leq f(x_0),$$ and $v\leq 0$ in $\mathbb{R}^{N}\setminus \Omega$. On the other hand, a lower semicontinuous
function $u: \mathbb{R}^{N} \to \mathbb{R}$ is a viscosity supersolution to \eqref{boundary problem} in $\Omega$ if $u \in L_{loc}(\mathbb{R}^{N})$, and  for any
open set $U \subset \Omega$, any $x_0 \in U$ and any $\psi \in \mathcal{C}^{2}(U)$ such that $u(x_0)=\psi(x_0)$ and $\phi \leq u$ in $U$, if we define $v$ as
\begin{equation}\label{function vv}
 v(x):= \left\lbrace
  \begin{array}{l}
  \psi(x) \text{ in } U\\
    u(x)  \text{ outside  }U, \\
  \end{array}
  \right.
\end{equation}there holds
$$(-\Delta)^{s}\psi(x_0) + |\nabla \psi(x_0)|^{p} \geq f(x_0)$$and $v\geq 0$ in $\mathbb{R}^{N}\setminus \Omega$. Finally, a viscosity solution to \eqref{boundary
problem} is a continuous function which is both a subsolution and a supersolution to \eqref{boundary problem}.
\end{definition}

To end this section, we prove the next non existence result that justifies in some way the condition $p<\frac{1}{1-s}$ that we will be used later.
\begin{theorem}\label{non}
Assume that $p>\frac{2s-1}{1-s}N+1$, then for all $0\lvertneqq f\in L^\infty(\O)$, problem \eqref{boundary problem} has no weak solution $u$ in the sense of Definition \ref{def1},  such that $u\in
W^{1,p}_0(\O)$.
\end{theorem}
\begin{proof}
Suppose by contradiction that problem \eqref{boundary problem} has a solution $u$ with $u\in W^{1,p}_0(\O)$. It is clear that $u$ solves the problem
$$
(-\D)^s u +\langle B(x),\n u\rangle =f,
$$
where $B(x)=|\n u|^{p-2}\n u$. Since $p>\frac{2s-1}{1-s}N+1$, then $|B|\in L^\s(\O)$ with $\s>\frac{N}{2s-1}$ and then $B\in \mathcal{K}^{s}_N(\O)$ the Kato class
of function defined  by formula (30) in \cite{BJ}. Thus
$$
u(x)=\io \hat{\mathcal{G}}_s(x,y)f(y)dy,
$$
where $\hat{\mathcal{G}}_s$ is the Green function associated to the operator $(-\D)^s +B(x)\n$. From the result of \cite{BJ}, we know that
$\hat{\mathcal{G}}_s\simeq \mathcal{G}_s$, the Green function associated to the fractional laplacian. Hence
$$
\mathcal{G}_s(x,y)\simeq C(B)\frac{1}{|x-y|^{N-2s}}\bigg(\frac{\d^s(x)}{|x-y|^{s}}\wedge 1\bigg) \bigg(\frac{\d^s(y)}{|x-y|^{s}}\wedge 1\bigg).
$$
Using the fact that $\dfrac{\d^s(x)}{|x-y|^{s}}\ge C(\O)\d^s(x)$, we reach that
$$
u(x)\ge C(B) \d^s(x)\io f(y) \d(y)\,dy.
$$
Therefore, using the Hardy inequality we deduce that
$$
\frac{\d^{sp}}{\d^p}\le C\frac{u^p}{\d^p}\in L^1(\O).
$$
Thus $\dfrac{1}{\d^{p(1-s)}}\in L^1(\O)$. Since $p(1-s)\ge 1$, then we reach a contradiction.
\end{proof}

\begin{corollary}
Let $f$ be a Lipschitz function such that $f\gneqq 0$, then if $p>\frac{1}{1-s}$, problem \eqref{boundary problem} has no solution $u$ such that $u\in
\mathcal{C}^1(\O)$ with $|\n u|\in L^p(\O)$.
\end{corollary}
\begin{remark}
It is clear that the above result makes a significative difference with the local case and the general existence result proved in \cite{PPL} for Lipschitz function. We conjecture that the non existence result holds at least for all $p>\frac{1}{s-1}$ as in the case of gradient reaction term.
\end{remark}

\section{Existence results.}\label{exis}

\subsection{The problem with natural growth in the gradient: $p<2s$.}\label{natural}
In this section we consider the case of natural growth in the gradient, namely we will assume that $p<2s$. Then using truncating arguments, we are able to show the
existence of a solution to problem  \eqref{boundary problem} for a large class of data. We also treat the case where a linear reaction term appears in
\eqref{boundary problem}.

In the case where $p<p_*$, then for more regular  data $f$, we can show that the solution is in effect a classical solution.
 \begin{Theorem}\label{regularity and order int f1}
 Let $f \in L^{m}(\Omega)$ with $m \geq 1$, and assume that $1 <p < p_{*}$. Then, the Dirichlet problem
  \begin{eqnarray*}
 \left\lbrace
  \begin{array}{l}
    (-\Delta)^{s}w+| \nabla w|^{p}=  f\quad  \text{ in } \Omega\\
  \qquad \,\,\,\quad \qquad \quad w=0 \quad \text{ in } \mathbb{R}^{N}\setminus \Omega, \\
  \end{array}
  \right.
\end{eqnarray*}
has a unique distributional solution $w$ verifying
  \begin{itemize}
 \item if $m < \frac{N}{2s-1}$, then $|\nabla w| \in L^{q}_{loc}(\Omega)$ for all $1\leq q < \frac{mN}{N-m(2s-1)}$;
 \item if $m= \frac{N}{2s-1}$, then $|\nabla w| \in L^{q}_{loc}(\Omega)$ for all $1 \leq q<\infty$;
 \item if $m > \frac{N}{2s-1}$, then $|\nabla w |\in \mathcal{C}^{\alpha}(\Omega)$ for some $\alpha \in (0,1)$.
\end{itemize}
Moreover, if in addition $f \in \mathcal{C}^{\epsilon}(\Omega)$, for some $\epsilon \in (0, 2s-1)$, then the $\mathcal{C}^{1, \alpha}$ distributional solution is
a strong solution.
\end{Theorem}

\begin{proof}
It is clear that the existence and the uniqueness follow using \cite{CV} and \cite{AP}, however, the regularity in the local Sobolev space follows using
Proposition \ref{key2-locc}. Notice that, in this case $|\n u|^{p-1}\in L^\s(\O)$ with $\s>\frac{N}{2s-1}$ and then we can iterate the local regularity result in
Proposition \ref{key2-locc} to deduce that $|\n u|\in L^\theta_{loc}(\O)$ for all $\theta>0$. Hence $|\n u|\in \mathcal{C}^{a}(\O)$ for some $a<1$.

Now, assume that $f \in \mathcal{C}^{\epsilon}(\Omega)$, and let $\Omega' \Subset \Omega$, open and let $u$ be a distributional solution to problem
\eqref{boundary problem}. Since $u \in L^{\infty}(\mathbb{R}^{N})$ and $f-|\nabla u|^{p} \in L^{\infty}(\Omega')$, we apply Proposition 2.3 in \cite{RS} to
derive
\begin{equation*}
u \in \mathcal{C}^{\beta}(\Omega''), \quad \textnormal{ for all }\beta \in (0, 2s), \, \Omega'' \Subset \Omega'.
\end{equation*}In particular,  we have $|\nabla u| \in \mathcal{C}^{\beta-1}(\Omega'')$ for any $\beta \in (1, 2s)$. Consequently, $f- |\nabla u|^{p} \in
\mathcal{C}^{\epsilon}(\Omega'')$. Appealing now to Corollary 2.4 in \cite{RS}, we obtain $u \in \mathcal{C}^{2s+\epsilon}$ in a smaller subdomain of $\Omega''$.
Thus, $u \in \mathcal{C}^{2s+\epsilon}$ locally in $\Omega$.

We prove that $u$ is a strong solution. Since the term $f-|\nabla u|^{p}$ is $\mathcal{C}^{\epsilon}$ in $\Omega$, and then, by appropriate extension, in
$\overline{\Omega}$, we deduce from \cite[Lemma 2.1(ii)]{CV1} that $u \in \mathbb{X}_s$. Hence the integration by parts formula
$$\int_\Omega u(-\Delta)^{s}\phi = \int_\Omega\phi(-\Delta)^{s}u $$holds for all $\phi \in \mathbb{X}_s$. For any $\phi \in \mathcal{C}^{\infty}_0(\Omega)$ we
hence obtain
\begin{equation*}
\int_\Omega\phi(-\Delta)^{s}u  = \int_\Omega u(-\Delta)^{s}\phi = \int_\Omega f\phi - \int_\Omega |\nabla u|^{p}\phi.
\end{equation*}Therefore
$$(-\Delta)^{s}u(x)=f(x)-|\nabla u(x)|^{p}$$for almost everywhere $x$ in $\Omega$. By continuity, it holds in the full set $\Omega$.
\end{proof}

\begin{remark}
Observe that the reasoning employed to prove the above result gives the precise way in which the function $f$ transfers its regularity to a solution $u$. Indeed,
if $f \in \mathcal{C}^{2ns+\epsilon -n}$ locally in $\Omega$, for $\epsilon \in (0, 2s-1)$ and $n \geq 0$, then $u \in \mathcal{C}^{2(n+1)s + \epsilon -n}$
locally in $\Omega$.
\end{remark}

\subsection{The case $p_*\le p<2s$ with general datum}\label{ex-gen}
In this subsection we will assume that $p_*\le p<2s$, then the first existence result for problem \eqref{boundary problem} is the following.
\begin{Theorem}\label{exist1}
Assume that $p<2s$, then for all $f\in L^1(\Omega)$ with $f\ge 0$, the problem \eqref{boundary problem} has a maximal weak solution $u$ such that $u\in
W^{1,p}_0(\O)$ and $T_k(u)\in W^{1,\a}_0(\Omega)\cap H^s_0(\Omega)$  for any $1<\a<2s$ and for all $k>0$.
\end{Theorem}
\begin{proof}
We divide the proof into two steps.

{\bf The first step:} We show for a fixed positive integer $n\in \ene^*$, the problem
\begin{equation}\label{aprox001} \left\{
\begin{array}{rcll}
(-\Delta)^s u_n +\dfrac{|\n u_n|^p}{1+\frac 1n |\n u_n|^p} &=& f & \text{   in }\Omega , \\  u_n &=& 0 &\hbox{  in }\mathbb{R}^N\setminus\Omega.
\end{array} \right.
\end{equation}
has a unique solution $u_n$ such that $u_n\in W^{1,q}_0(\O)$ for all $1\le q<\frac{N}{N-2s+1}$ and $T_k(u_n)\in H^s_0(\O)$. To prove that, we proceed by approximation.

Let $k\in \ene^*$ and define $u_{n,k}$ to be the unique solution to the approximating problem
\begin{equation}\label{aprox1k} \left\{
\begin{array}{rcll}
(-\Delta)^s u_{n,k} +\dfrac{|\n u_{n,k}|^p}{1+\frac 1n |\n u_{n,k}|^p} &=& f_k & \text{   in }\Omega , \\  u_{n,k} &=& 0 &\hbox{  in }\mathbb{R}^N\setminus\Omega.
\end{array} \right.
\end{equation}
where $f_k=T_k(f)$. We claim that the sequence $\{u_{n,k}\}_k$ is increasing in $k$, namely $u_{n,k}\le u_{n,k+1}$ for all $k\ge 1$ and $n$ fixed. To see that, we have
$$
(-\D)^s u_{n,k+1} + \dfrac{|\n u_{n,k+1}|^p}{1+\frac 1n |\n u_{n,k+1}|^p}\ge f_k.
$$
Thus $u_{n,k+1}$ is a supersolution the problem solved by $u_{n,k}$. Setting $H(x,s,\xi)= -\dfrac{|\xi|^p}{1+\frac 1n |\xi|^p}$, then
by the comparison principle in Theorem \ref{compa2}, it follows that $u_{n,k}\le u_{n,k+1}$ and then the claim follows. It is clear that $u_{n,k}\le w$ for all $n, k\in \ene^*$ where $w$ is the unique solution to
problem
\begin{equation}\label{inter} \left\{
\begin{array}{rcll}
(-\Delta)^s w &= & f &\text{   in }\Omega , \\  w &=& 0 &\hbox{  in } \mathbb{R}^N\setminus\Omega.
\end{array} \right.
\end{equation}
Notice that $w\in W^{1,q}_0(\O)$ for all $1\le q<\frac{N}{N-2s+1}$ and $w\in L^r(\O)$ for all $1\le r<\frac{N}{N-2s}$.

Hence, we get the existence of $u_n$ such that $u_{n,k}\uparrow u_n$ strongly in $L^{\s}(\O)$ for all $1\le \s<\frac{N}{N-2s}$.

For $n$ fixed, we set $h_{n,k}:=f_{k}-\dfrac{|\n u_{n,k}|^p}{1+\frac 1n |\n u_{n,k}|^p}$, then $|h_{n,k}|\le f+n$. Thus
$||h_{n,k}||_{L^1(\O)}\le ||f||_{L^1(\O)}+n|\O|$. Hence using the compactness result in Theorem \ref{key} we deduce that
up to a subsequence, $u_{n,k}\to u_n$ strongly in $W^{1,\a}_0(\O)$ for all $\a<\frac{N}{N-2s+1}$. Since the sequence $\{u_{k,n}\}_k$ is increasing in $k$, then the limit $u_n$ is unique. Thus, up to a further subsequence, $\n
u_{n,k}\to \n u_n$ a.e. in $\O$. Hence using the dominated convergence Theorem it holds that
$$
\dfrac{|\n u_{n,k}|^p}{1+\frac 1n |\n u_{n,k}|^p}\to \dfrac{|\n u_{n}|^p}{1+\frac 1n |\n u_{n}|^p}\mbox{  strongly in }L^a(\O)\mbox{  for all }a<\infty.
$$
Hence $u_n$ solves the problem \eqref{aprox001}. To proof the uniqueness of $u_n$, we assume that $v_n$ is another solution to problem \eqref{aprox001}, then
$$
(-\D)^s(u_n-v_n)=\hat{H}(|\n u_n|)-\hat{H}(|\n v_n|),
$$
where $\hat{H}(|\xi|)=-\dfrac{|\xi|^p}{1+\frac 1n |\xi|^p}$. Since $|\hat{H}(|\xi|)|\le n$, then we obtain that $u_n-v_n\in L^\infty(\O)$. Finally using the comparison principle in  Theorem \ref{compa2} it follows that $u_n=v_n$ and then we conclude.

{\bf Second step:} Consider the sequence $\{u_n\}_n$ obtained in the first step, then we know that $u_n\le w$ for all $n$. We claim that $u_n$ is decreasing in $n$. Recall that $u_n$ is the unique solution to the problem
\begin{equation}\label{aprox1} \left\{
\begin{array}{rcll}
(-\Delta)^s u_n +\dfrac{|\n u_n|^p}{1+\frac 1n |\n u_n|^p} &=& f & \text{   in }\Omega , \\  u_n &=& 0 &\hbox{  in }\mathbb{R}^N\setminus\Omega.
\end{array} \right.
\end{equation}
Thus
$$
(-\Delta)^s u_n +\dfrac{|\n u_n|^p}{1+\frac 1{n+1} |\n u_n|^p}\ge f.
$$
Hence $u_n$ is a supersolution to the problem solved by $u_{n+1}$. As a consequence and using the comparison principle in Theorem \ref{compa2}, it follows that $u_{n+1}\le u_{n}\le w$ for all $n$.

Hence,  there exists $u$ such that $u_n\downarrow u$ strongly in $L^{\s}(\O)$ for all $1\le \s<\frac{N}{N-2s}$.

We set $g_n(|\n u_n|)=\dfrac{|\n u_n|^p}{1+\frac 1n |\n u_n|^p}$, and let $j>0$, using $T_j(u_n)$ as a test function in \eqref{aprox1} it follows that
$$
\dyle \iint_{D_\O}\frac{(T_{j}(u_n(x))-T_{j}(u_n(y)))^2}{|x-y|^{N+2s}}\, dxdy +\io g_n(|\n u_n|) T_j(u_n) dx \leq Cj\dyle.
$$
Hence $\{T_j(u_n)\}_n$ is bounded in  $H^s_0(\Omega)$ for all $j>0$ and then, up to a subsequence, we have $T_j(u)\rightharpoonup T_j(u)$ weakly in $H^s_0(\Omega)$.
We claim that $\{g_n\}_n$ is bounded in $L^1(\O)$. To see that, we fix $\e>0$ and we use $v_{n,\e}=\frac{u_n}{\e+u_n}$ as a test function in \eqref{aprox1}. It is
clear that $v_{n,\e}\le 1$, then taking into consideration that
$$
(u_n(x)-u_n(y))(v_{n,\e}(x)-v_{n,\e}(y))\ge 0,
$$
it follows that
$$
\io g_n(|\n u_n|)v_{n,\e}(x)\le \io f dx\le C.
$$
Letting $\e\to 0$, we reach that $\io g_n(|\n u_n|)dx\le C$ an the claim follows. Define $h_n=f-g_n$, then $||h_n||_{L^1(\O)}\le C$. As a consequence and by the
compactness result in Theorem \ref{key}, we reach that, up to a subsequence, $u_n\to u$ strongly in $W^{1,\a}_0(\O)$ for all $1\le \a<\frac{N}{N-2s+1}$ and then,  up to an other subsequence, $\n
u_n\to \n u$ a.e in $\O$. Hence $g_n\to g$ a.e. in $\O$ where $g(x)=|\n u|^p$. Since $p<2s$, then by Theorem \ref{key01} and using Vitali Lemma we conclude
that
$$
T_k(u_n)\to T_k(u)\mbox{  strongly in } W^{1,\s}_0(\O) \mbox{  for all }\s<2s.
$$
In particular
\begin{equation}\label{conv1}
T_k(u_n)\to T_k(u)\mbox{  strongly in } W^{1,p}_0(\O).
\end{equation}
Hence to get the existence result we have just to show that $g_n\to g$ strongly in $L^1(\O)$.

Notice that, using $T_1(G_j(u_n))$ as a test function in \eqref{aprox1} it holds that
$$
\int_{u_n\ge j+1}g_ndx\le \int_{u_n\ge j}f dx\to 0\mbox{ as  }j\to \infty.
$$
Let $\e>0$ and consider $E\subset \O$ to be a measurable set, then
\begin{eqnarray*}
\dyle \int_E g_n dx &= & \int_{\{E\cap \{u_n<j+1\}\}} g_n dx +\int_{\{E\cap \{u_n\ge j+1\}\}} g_n dx\\ &\le & \dyle \int_{\{E\cap \{u_n<j+1\}\}} |\n
T_{j+1}(u_n)|^p dx + \int_{\{u_n\ge j+1\}} fdx.
\end{eqnarray*}
By \eqref{conv1}, letting $n\to \infty$, we can chose $|E|$ small enough such that
$$
\limsup_{n\to \infty}\int_{\{E\cap \{u_n<j+1\}\}} |\n T_{j+1}(u_n)|^p dx\le \frac{\e}{2}.
$$
In the same way and since $f\in L^1(\O)$, we reach that
$$
\limsup_{n\to \infty}\int_{\{u_n\ge j+1\}} fdx\le \frac{\e}{2}.
$$
Hence, for $|E|$ small enough, we have
$$
\limsup_{n\to \infty}\int_E g_n dx\le \e.
$$
Thus by Vitali lemma we obtain that $g_n\to g$ strongly in $L^1(\O)$. Therefore we conclude that $u$ is a solution to problem \eqref{boundary problem}.

If $\hat{u}$ is an other solution to \eqref{boundary problem}, then
$$
(-\Delta)^s \hat{u} +\dfrac{|\n \hat{u}|^p}{1+\frac 1n |\n \hat{u}|^p}\le f.
$$
Hence $\hat{u}\le u_n$ and then $\hat{u}\le u$.
\end{proof}
\begin{remark}

\

\begin{enumerate}
\item The existence of a unique solution to the approximating problem \eqref{aprox1} holds for all $p\ge 1$.
\item Problem of uniqueness of solution to problem \eqref{boundary problem} is an interesting open problem including for the local case $s=1$ where partial results are known in the case $1<p<\frac{N}{N-1}$ or $p=2$.

\item As a consequence of the previous result and
    following closely the same argument we can prove that for all $p<2s$, for all $a>0$ and for all $(f,g)\in L^1(\O)\times L^1(\O)$ with $f,g\gneqq 0$, the
    problem \begin{equation}\label{aprox12} \left\{
\begin{array}{rcll}
(-\Delta)^s u+|\n u|^p &=& g(x)\dfrac{u}{1+a u}+ f & \text{   in }\Omega , \\  u &=& 0 &\hbox{  in }\mathbb{R}^N\setminus\Omega.
\end{array} \right.
\end{equation}
has a positive solution $u$.
\end{enumerate}
\end{remark}

In the case where the datum $f$ is substituted by a Radon measure $\nu$, existence of solutions holds for all $1< p<p_*$ as it was proved in \cite{CV}. However, if $p\ge p_*$, then the situation changes completely as in the local case, and, additional hypotheses on $\nu$ related to a fractional capacity $\text{Cap}_{\s,p}$ are needed, with $\s<1$.

The fractional capacity $\text{Cap}_{\s,p}$ is defined as follow.

For a compact set $K\subset \O$, we define
\begin{equation}\label{compact}
\text{Cap}_{\s,p}(K) =\inf \left\{\|\psi\|_{W^{\s,p}_0(\O)} : \psi \in W^{\s,p}_0(\O), 0\le \psi\le 1 \mbox{  and  }
\ \psi \geq \chi_{K} \ {\rm a.e.\ in}\quad
\O\right\}.
\end{equation}
Now, if $U\subset \O$ is an open set, then
\begin{equation*}\label{open2}
\text{Cap}_{\s,p}(U) =\sup \left\{\text{Cap}_{\s,p}(K): K\subset U \mbox{  compact of } \O \mbox{  with  }K\subset U\right\}.
\end{equation*}
For  any borel subset $B\subset \O$, the definition is
extended by setting:
\begin{equation*}
\text{Cap}_{\s,p}(B)= \inf\left\{\text{Cap}_{\s,p}(U),\ U\mbox{ open subset of
$\O$, }B \subset U\right\}.
\end{equation*}
Notice that, using Sobolev inequality, we obtain that if $\text{Cap}_{\s,p}(A)=0$ for some set $A\subset \subset \O$, then $|A|=0$.
We refer to \cite{War} for the main properties of this capacity.

To show that the situation changes for the set of general Radon measure, we prove the next non existence result.
\begin{Theorem}\label{mes}
Assume that $p>p_*$, $\frac 12<s<1$ and let $x_0\in \O$, then the problem
\begin{equation}\label{delta}\left\{
\begin{array}{rcll}
(-\Delta)^s u + |\grad u|^{p}&=&\d_{x_0} & \text{   in }\Omega ,
\\  u &=& 0 &\hbox{  in }\mathbb{R}^N\setminus\Omega,
\end{array} \right.
\end{equation}
has non solution $u$ such that $u\in W^{1,p}_0(\O)$.
\end{Theorem}
\begin{proof}
For simplify of tipping we assume that $x_0=0\in \O$ and we write $\d$ for $\d_0$. We follow closely the argument used in \cite{BP}. Assume by contradiction that for some $p>p_*$, problem \eqref{delta} has a solution $u\in W^{1,p}_0(\O)$. Then $u\in W^{\s,p}_0(\O)$ for all $\s<1$. We claim that $(-\Delta)^s u \in W^{-\s,p}(\O)$, the dual space of $W^{\s,p}_0(\O)$, for all $\s\in (2s-1,2s)$. To see that, we consider $\phi\in \mathcal{C}^\infty_0(\O)$, then
\begin{eqnarray*}
|\io (-\Delta)^s u \phi dx| &\le & \iint_{\re^{2N}} \dfrac{|u(x)-u(y)||\phi(x)-\phi(y)|}{|x-y|^{N+2s}}\,dxdy\\
&\le & \bigg(\iint_{\re^{2N}} \dfrac{|u(x)-u(y)|^p}{|x-y|^{N+p(2s-\s)}}\,dxdy\bigg)^{\frac{1}{p}}  \bigg(\iint_{\re^{2N}} \dfrac{|\phi(x)-\phi(y)|^{p'}}{|x-y|^{N+p' \s}}\,dxdy\bigg)^{\frac{1}{p'}}.
\end{eqnarray*}
Since $2s-\s\in (0,1)$, then $\dyle\bigg(\iint_{\re^{2N}} \dfrac{|u(x)-u(y)|^p}{|x-y|^{N+p(2s-\s)}}\,dxdy\bigg)^{\frac{1}{p}} \le C(\s,s,N,\O)||u||_{W^{1,p}_0(\O)}$. Thus
$$
|\io (-\Delta)^s u \phi dx| \le C ||u||_{W^{1,p}_0(\O)} ||\phi||_{W^{\s,p'}_0(\O)},
$$
and then the claim follows. Hence going back to problem \eqref{delta}, we deduce that $\d\in L^1(\O)+W^{-\s,p}(\O)$.

As in \cite{BGO}, let us now show that if $\nu\in W^{-\s,p}(\O)$, then $\nu << \text{Cap}_{\s,p'}$.  Notice that, if in addition, $\nu$ is nonnegative, then we can prove that
\begin{equation*}\label{contr}
\nu (A)\le C (\text{Cap}_{\s,p'}(A))^{\frac 1p},
\end{equation*}
and we deduce easily that $\nu << \text{Cap}_{\s,p'}$. Here we give the proof without the positivity assumption on $\nu$.

Let $A\subset\subset \O$ be such that $\text{Cap}_{\s,p'}(A)=0$, then  there exists a Borel set $A_0$ such that $A\subset A_0$ and $\text{Cap}_{\s,p'}(A_0)=0$. Let $K\subset A_0$ be a compact set, then there exists a sequence $\{\psi_n\}_n\in \mathcal{C}^\infty_0(\O)$ such that $0\le \psi_n\le 1$, $\psi_n\ge \chi_K$ and $||\psi_n||^{p'}_{W^{\s,p'}_0(\O)}\to 0$ as $n\to \infty$. It is clear that $\psi_n\to \chi_K$ a.e in $\O$, as $n\to \infty$. Hence
$$
\nu (K)=\lim_{n\to \infty}\int \psi_n d\nu=\lim_{n\to \infty}\langle \psi_n, \nu\rangle_{W^{\s,p'}_0(\O), W^{-\s,p}_0(\O)}.
$$
Thus
$$
|\nu (K)|\le \limsup_{n\to \infty}|\langle \psi_n, \nu\rangle_{W^{\s,p'}_0(\O), W^{-\s,p}_0(\O)}|\le
 \limsup_{n\to \infty} ||\nu||_{W^{-\s,p}_0(\O)} ||\psi_n||_{W^{\s,p'}_0(\O)}=0.
$$
Therefore, we conclude that for any compact set $K\subset A_0$, we have $|\nu (K)|=0$. Hence $|\nu (A_0)|=0$ and the result follows.

Notice that if $h\in L^1(\O)$, then $|h|<< \text{Cap}_{\s,p'}$. As a conclusion, we deduce that $\d<< \text{Cap}_{\s,p'}$ for all $\s\in (2s-1,2s)$.

Since $p>p_*$,  we can choose $\s_0\in (2s-1,2s)$ such that $p'\s_0<N$. To end the proof, we have just to show that $\text{Cap}_{\s_0,p'}\{0\}=0$. Without loss of generality, we can assume that $\O=B_1(0)$. Since $\s p'<N$, setting $w(x)=(\dfrac{1}{|x|^\a}-1)_+$ with $0<\a<\frac{N-\s_0 p'}{p'}$, we obtain that $w\in W^{\s,p'}_0(\O)$.
Notice that, for all $v\in W^{\s,p'}_0(\O)$, we know that
$$
\text{Cap}_{\s,p'}\{|v|\ge k\}\le \frac{C}{k}||v||_{W^{\s,p'}_0(\O)}.
$$
Since $w(0)=\infty$, then $\{0\}\subset \{|w|\ge k\}$ for all $k>0$. Thus
$$
\text{Cap}_{\s,p'}\{0\}\le \frac{C}{k}||w||_{W^{\s,p'}_0(\O)}\mbox{  for all }k.
$$
Letting $k\to \infty$, it holds that $\text{Cap}_{\s,p'}\{0\}=0$ and the result follows.
\end{proof}

As a direct consequence of the above Theorem we obtain that for $p>p_*$, to get the existence of a solution to problem \eqref{boundary problem} with measure data $\nu$, then necessarily $\nu$ is continuous with respect to the capacity $\text{Cap}_{\s,p}$ for all $\s\in (2s-1,2s)$.

\

\

Let consider now the next problem
\begin{equation}\label{eq:prob}\left\{
\begin{array}{rcll}
(-\Delta)^s u + |\grad u|^{p}&=&\l g(x)u+f & \text{   in }\Omega ,
\\  u &=& 0 &\hbox{  in }\mathbb{R}^N\setminus\Omega.
\end{array} \right.
\end{equation}
with $g\gneqq 0$ and $\lambda >0$. As in local case studied in \cite{APP}, we can show that under natural conditions on $q$ and $g$, the problem \eqref{eq:prob} has a solution for
all $\l>0$. Moreover, the gradient term $|\nabla u|^q$ produces a strong regularizing effect on the problem and kills any effect of the linear term $\l g u$.

Before stating the main existence result for problem \eqref{eq:prob}, let us begin by the next definition.

\begin{definition}\textit{Let $g$ be a nonnegative measurable function such that $g\in L^1(\O)$. We say that $g$ is an {\it  admissible weight} if
\begin{equation}\label{eq:C}
C(g,p)=\inf\limits_{\phi\in W_{0}^{1,p}(\O)\setminus\{0\}}\dfrac{\left(\dint_\Omega\,|\grad \phi|^{p}\,dx\right)^{\frac{1}{p}}}{\dint_\Omega\,g|\phi|\,dx}>0.
\end{equation}}
\end{definition}

\noindent $\bullet$ If $g \in L^{\frac{p N}{N(p-1)+p}}(\Omega)$ with $g\gneqq 0$, then using the Sobolev inequality in the space $W^{1,p}_0(\O)$, it holds that $g$ satisfies \eqref{eq:C}.

\noindent $\bullet$ If $p<N$ and $g(x)=\frac{1}{|x|^\s}$ with $\s<1+\frac{N}{p'}$, then using the Hardy-Sobolev inequality in the space $W^{1,p}_0(\O)$, we deduce that $g$ satisfies \eqref{eq:C}.

\

Now, we are able to state the next result.
\begin{Theorem}\label{exist12}
Assume that $1<p<2s$ and suppose that $g$ is an admissible weight in the sense given in \eqref{eq:C}. Then for all $f\in L^1(\Omega)$ with $f\ge 0$ and for all
$\l>0$, the problem \eqref{eq:prob} has a  solution $u$ such that $u\in W^{1,p}_0(\O)$ and $T_k(u)\in W^{1,\a}_0(\Omega)\cap H^s_0(\Omega)$  for any $1<\a<2s$ and for all
$k>0$.
\end{Theorem}
\begin{proof}
Fix $\l>0$ and define $\{u_n\}_n$ to be a sequence of positive solutions to problem
\begin{equation}\label{aprox122} \left\{
\begin{array}{rcll}
(-\Delta)^s u_n+|\n u_n|^p &=& \l g(x)\dfrac{u_n}{1+\frac 1n u_n}+ f & \text{   in }\Omega , \\  u_n &=& 0 &\hbox{  in }\mathbb{R}^N\setminus\Omega.
\end{array} \right.
\end{equation}
To reach the desired result we have just to show that the sequence $\{g(x)\dfrac{u_n}{1+\frac 1n _nu}\}_n$ is uniformly bounded in $L^1(\O)$. To do that, we use
$T_k(u_n)$ as a test function in \eqref{aprox122}, hence
\begin{equation}\label{estim-p}
||T_k(u_n)||^2_{H^s_0(\Omega)}+ \io |\n u_n|^p T_k(u_n) dx \le k \l \io g(x) u_n dx + k\| f\|_{L^{1}(\Omega)} .
\end{equation}
It is clear that
$$\io |\n u_n|^p T_k(u_n)=\io |\n H_k(u_n)|^p dx$$
where $H_k(\s)=\dint_0^\s (T_k(t))^{\frac{1}{p}}dt$. By a direct computation we obtain that
$$
H_k(\s)\ge C_1(k)\s-C_2(k),
$$
Thus using \eqref{eq:C} for $H_k(u_n)$ it holds that
\begin{eqnarray*}
\dyle \io |\n H_k(u_n)|^p dx &\ge & C(g,p)\bigg(\io g H_k(u_n) dx\bigg)^p\\ &\ge &  C(g,p)\left[C_1(k) \bigg(\io g u_n dx\bigg)^p -C_2(k)\right]
\end{eqnarray*}
where $C_1(k),C_2(k)>0$ are independent of $n$.

Therefore, going back to \eqref{estim-p}, we conclude that
$$
C_1(k) \bigg(\io g u_n dx\bigg)^p dx\le \frac{k}{C(g,p)}\left[\lambda\io g(x) u_n dx + \| f\|_{L^{1}(\Omega)}\right] + C_2(k).
$$
Since $p>1$, then by Young inequality we reach that $\{g u_n\}_n$ is uniformly bounded in $L^1(\O)$. The rest of the proof follows exactly the same compactness
arguments as in the proof of Theorem \ref{exist1}.
\end{proof}
\begin{remark}

\

\noindent $\bullet$ In the case where $g(x)=\dfrac{1}{|x|^{2s}}$, the Hardy potential, the condition \eqref{eq:C} holds if $p>\frac{N}{N-(2s-1)}$. Thus, in this case and for all
$\l>0$, problem \eqref{eq:prob} has a solution $u$ such that $u\in W^{1,p}_0(\O)$ and $T_k(u)\in W^{1,\a}_0(\Omega)\cap H^s_0(\Omega)$ for all $\a<2s$.

\noindent $\bullet$ Notice that, in this case, without the absorption term $|\n u|^p$, the existence of solution holds under the restriction $\l\le \Lambda_{N,s}$, where $\Lambda_{N,s}$ is the Hardy constant, and with integral condition on the datum $f$ near the origin. We refer to \cite{AMPP} for more details.
\end{remark}

\subsection{The case $2s\le p<\frac{s}{1-s}$: existence in a weighted Sobolev space.  }\label{general}

\

For $2s\le p<\dfrac{s}{1-s}$ and in the same way as above we can show the next existence result.
\begin{theorem}\label{thh}
Suppose that $f \in L^{m}(\O)$ with $m > N/[p'(2s-1)]$. Then there is $\lambda^*>0$ such that if $||f||_{L^m(\O)}\le \l^*$, problem \eqref{boundary problem}
admits a solution $u\d^{1-s}\in W_0^{1, p}(\Omega)$.
\end{theorem}
\begin{proof}
The proof follows closely the argument used in \cite{AP}, however, for the reader convenience we include here some details.

Without loss of generality we can assume that $N\ge 2$. Fix $\l^*>0$ such that if $||f||_{L^m(\O)}\le \l^*$, then there exists $l>0$ satisfies
$$
\bar{C}(\O,N,s,m,p)(l+||f||_{L^m(\O)})=l^{\frac{1}{p}},
$$
where $\bar{C}(\O,N,s,m,p)$ is a positive constant which only depends on the data, it is independent of $f$ and its will be specified below.

Define now the set
\begin{equation}\label{sett}
E=\bigg\{v\in W^{1,1}_0(\O): v\, \d^{1-s}\in W^{1,pm}_0(\O)\mbox{  and  } \bigg(\io |\n (v\, \d^{1-s})|^{pm} dx\bigg)^{\frac{1}{pm}}\le l^{\frac{1}{2s}}\bigg\},
\end{equation}
It is clear that $E$ is a closed convex set of $W^{1,1}_0(\O)$. Using Hardy inequality in \eqref{hardydye}, we deduce that if $v\in E$, then $|\n v|^{pm}\d^{pm(1-s)}\in L^1(\O)$ and
$$
\bigg(\io |\n v|^{pm}\, \d^{pm(1-s)}dx\bigg)^{\frac{1}{pm}}\le \hat{C}_0(\O)l^{\frac{1}{p}}.
$$
Define now the operator
$$
\begin{array}{rcl}
T:E &\rightarrow& W^{1,1}_0(\O)\\
   v&\rightarrow&T(v)=u
\end{array}
$$
where $u$ is the unique solution to problem
\begin{equation}
\left\{
\begin{array}{rcll}
(-\Delta)^s u &= & -|\nabla v|^{p}+f & \text{ in }\Omega , \\ u &=& 0 &\hbox{  in } \mathbb{R}^N\setminus\Omega,\\ u&>&0 &\hbox{ in }\Omega.
\end{array}%
\right.  \label{fix1}
\end{equation}
To prove that $T$ is well defined we will use Theorem \ref{key}, namely we show the existence of $\b<2s-1$ such that $\bigg|f-|\n v|^{p}\bigg|\d^\b\in L^1(\O)$. To do that we have just to show that $|\n v|^{p}\d^\b\in L^1(\O)$.

It is cleat
that $|\n v|^{p}\in L^1_{loc}(\O)$, moreover, we have
$$
\io|\n v|^{p}\d^\b dx=\io|\n v|^{p}\d^{p(1-s)}\d^{\beta-p(1-s)}dx\le \bigg(\io|\n v|^{pm}\d^{pm(1-s)}dx\bigg)^{\frac{1}{m}} \bigg(\io
\d^{(\beta-p(1-s))m'}dx\bigg)^{\frac{1}{m'}}.
$$
If $p(1-s)<2s-1$, we can chose $\beta<2s-1$ such that $p(1-s)<\beta$. Hence $\io \d^{(\beta-p(1-s))m'}dx<\infty$.

Assume that $p(1-s)\ge 2s-1$, then $s\in (\frac{1}{2}, \frac{p+2}{p+1}]$. Notice that, since $p<\frac{s}{1-s}$, then $p(1-s)-(2s-1)<1-s$. Since $m>\frac{N}{p'(2s-1)}>\frac{1}{s}$, then
$(p(1-s)-(2s-1))m'<1$. Hence we get the existence of $\beta<2s-1$ such that $(p(1-s)-\beta)m'<1$ and then we conclude.

Then using the fact that $v\in E$, we reach that $|\nabla v|^{p}\d^\b+f \in L^1(\O)$. Therefore the existence of $u$ is a consequence of Theorems \ref{key} and \ref{th AP}. Moreover,  $|\n u|\in L^{\a}(\O)$ for all $\a<\frac{N}{N-2s+1+\beta}$. Hence $T$ is well defined.

Now following the argument used in \cite{AP}, for $l$ defined as above and using the regularity result in Theorem \ref{th AP} where we choose the constant $\bar{C}$ strongly related to the constant $\hat{C}$ defined in formula \eqref{dd11-corr}, we can prove that $T$ is continuous and compact on $E$ and that $T(E)\subset E$. For the reader convenience we included some details

We have
$$
u(x)=\io \mathcal{G}_s(x,y)f(y))dy- \io \mathcal{G}_s(x,y)|\n v(y)|^{2s}dy,
$$
then
$$
\n u(x)=\io \n_x\mathcal{G}_s(x,y)f(y))dy- \io \n_x\mathcal{G}_s(x,y)|\n v(y)|^{p}dy.
$$
Thus
$$
|\n u(x)|\le\io \frac{|\n_x \mathcal{G}_s(x,y)|}{\mathcal{G}_s(x,y)}\mathcal{G}_s(x,y) (|\n v(y)|^{p} dy+f(y))dy.
$$
Taking into consideration the properties of the Green function, it holds that
\begin{eqnarray*}
|\n u(x)| \d^{1-s}&\le & C_2(\O,N,s)(I_1(x)+I_2(x)+J_1(x)+J_2(x)),
\end{eqnarray*}
where
$$
I_1(x)=\d^{1-s}(x)\int_{\{|x-y|<\d(x)\}} \frac{\mathcal{G}_s(x,y)}{|x-y|}|\n v(y)|^{2s} dy,
$$
$$
I_2(x)=\frac{1}{\d^s(x)} \int_{\{|x-y|\ge \d(x)\}} \mathcal{G}_s(x,y)|\n v(y)|^{2s} dy,
$$
$$
J_1(x)=\d^{1-s}(x)\int_{\{|x-y|<\d(x)\}} \frac{\mathcal{G}_s(x,y)}{|x-y|}f(y)dy,
$$
and
$$
J_2(x)=\frac{1}{\d^s(x)} \int_{\{|x-y|\ge \d(x)\}} \mathcal{G}_s(x,y)f(y) dy.
$$
Following the arguments used in \cite{AP} and using the regularity result in Theorem \ref{key}, we get the existence of a positive constant $\breve{C}:=\breve{C}(\O,N,s,m)$ such that $I_i,J_i\in L^{pm}(\O)$ for $i=1,2$ and $$
||I_i||_{L^{pm}(\O)} +||J_i||_{L^{pm}(\O)}\le \breve{C}\bigg(|||\n v|\d^{1-s}||_{L^{pm}(\O)}+ ||f||_{L^m(\O)}\bigg).
$$
Hence assuming that $\bar{C}(\O,N,s,m,p)=\breve{C}(\O,N,s,m,p)C_2(\O,N,s,m,p)$, it follows that
\begin{eqnarray*}
|||\n v|\d^{1-s}||_{L^{pm}(\O)} & \le & \breve{C}(\O,N,s,m,p)C_2(\O,N,s,m,p)\bigg(|||\n v|\d^{1-s}||_{L^{pm}(\O)}+ ||f||_{L^m(\O)}\bigg)\\
&\le & \bar{C}(\O,N,s,m,p)(l^{\frac{1}{p}}+||f||_{L^m(\O)}\bigg)=l^{\frac{1}{p}}.
\end{eqnarray*}
Thus $u\in E$ and then $T(E)\subset E$. In the same way we can prove that $T$ is compact.

Therefore by the Schauder Fixed Point Theorem, there exists $u\in E$ such that $T(u)=u$. Thus,  $u\in
W^{1,pm}_{loc}(\O)$ solves \eqref{boundary problem}, at least in the sense of distribution.
\end{proof}

\begin{remark}

\

\begin{enumerate}
\item It is clear that the above argument does not take advantage of the fact that the gradient term appears as an absorption term.

\item The existence of a solution can be also proved independently of the sign of $f$.
\end{enumerate}
\end{remark}
As in Theorem \ref{regularity and order int f1}, if in addition we suppose that $f$ is more regular, then under suitable hypothesis on $s$ and $p$, we get the
following analogous result of Theorem \ref{regularity and order int f1}.

\begin{corollary}\label{cor classic}
Assume that the conditions of Theorem \ref{thh} hold. Assume in addition that
\begin{equation}\label{more}
N<\frac{s(2s-1)}{1-s}\mbox{   and   }p<\frac{s(2s-1)}{N(1-s)}-1.
\end{equation}
If $f \in \mathcal{C}^{\epsilon}(\Omega)$, for some $\epsilon \in (0, 2s-1)$, then the $\mathcal{C}^{1, \alpha}$ distributional solutions from Theorem
\ref{thh} is a strong solution.
\end{corollary}
Notice that the condition \eqref{more} is used in order to show that $|\n u|^{p-1}\in L^\s_{loc}(\O)$ for some $\s>\frac{N}{2s-1}$ which is the key point in order
to get the desired regularity.

In the case where $f\gvertneqq 0$, we can prove also that $u\gneqq 0$, more precisely, we have
\begin{corollary}\label{Non-negative}
Assume that the above conditions hold. Let $f \in \mathcal{C}^{\epsilon}(\Omega)\cap L^{\infty}(\overline{\Omega})$, for some $\epsilon \in (0, 2s-1)$. If
$f(x)\geq 0$ for all $x \in \Omega$, then the solution  from Theorem \ref{thh} is non-negative. Moreover, if $f_1 \leq f_2$ and $u_1$ and $u_2$ are the
corresponding strong solutions to $f_1$ and $f_2$ from Corollary \ref{cor classic}, respectively, then $u_1 \leq u_2$.
\end{corollary}
\begin{proof}
Suppose that there is a point $x_0 \in \Omega$ so that $u(x_0) <0$. Since $u$ is continuous in $\mathbb{R}^{N}$ (see Proposition 1.1 in \cite{RS}), we have $u$
attains its negative minimum at an interior point $x_1$ of $\Omega$. Hence
$$\nabla u(x_1)=0, \quad (-\Delta)^{s}u(x_1)< 0.$$But hence we obtain the contradiction $0 \leq f(x_1)-0 = (-\Delta)^{s}u(x_1)< 0$.

We next prove the last statement in the Corollary \ref{Non-negative}. Let $f_1 \leq f_2$. Let $u_1$ and $u_2$ be the corresponding strong solutions  from Corollary \ref{cor classic}, and assume that
$$\min_{\Omega}(u_2-u_1)=u_2(x_0)-u_1(x_0) < 0.$$Hence $\nabla(u_1-u_2)(x_0)=0$ and $(-\Delta)^{s}(u_2-u_1)(x_0)<0$, so we have the contradiction
$$f_1(x_0)=(-\Delta)^{s}u_1(x_0) +|\nabla u_1(x_0)|^{p} > (-\Delta)^{s}u_2(x_0) +|\nabla u_2(x_0)|^{p} = f_2(x_0).$$
\end{proof}

\section{Equivalence between distributional and viscosity solutions}\label{equiv}

In this section, we investigate the relation between distributional solutions and viscosity solutions. Let us recall that according to Theorem \ref{regularity and
order int f1} and  Corollary \ref{cor classic}, to obtain strong solutions to \eqref{boundary problem} it is sufficient that $f \in
\mathcal{C}^{\epsilon}(\Omega)$ and that
$$p < p_{*}$$ or
$$ N<\frac{s(2s-1)}{1-s},\:\:p^*\le p<\frac{s(2s-1)}{N(1-s)}-1\hbox{   and  } ||f||_{L^{m}(\Omega)} \leq \l^*,$$ for $\l^*$ defined in Theorem \ref{thh}. In this
section we show that strong solutions to \eqref{boundary problem} are viscosity solutions. The converse is also true provided a comparison principle for viscosity
solutions. We  prove it in the next subsection.

\subsection{A comparison principle for viscosity solutions.}

 We prove a comparison result for viscosity solutions of problem \eqref{boundary problem}. This result requires a continuous source term $f$.

In order to state the result, we shall need some technical lemmas that could have interest by themselves.  For related results see \cite{Li}.

We start with a usual property for the fractional Laplacian of smooth functions. See \cite[Lemma 2.6]{KKL} for the proof.

\begin{lemma}\label{absolute continuity}
Let $B_\epsilon(x)\subset U \Subset\Omega$ and let $u \in \mathcal{C}^{2}(U)$. Then:
$$\Big\vert P.V. \int_{B_\epsilon(x)}\frac{u(x)-u(y)}{|x-y|^{N+2s}}dy \Big\vert \leq c_\epsilon$$where $c_\epsilon$ is independent of $x$ and $c_\epsilon \to 0$
as $\epsilon \to 0$.
\end{lemma}

Observe that in the  definition of viscosity solutions, we do not evaluate  the given equation  in the solution $u$. However, the following lemma states  an extra information when
$u$ is touched from below or above by $C^{2}$-test functions.

\begin{lemma}\label{smooth functions}
Let $u$ be a viscosity supersolution to \eqref{boundary problem} and suppose that there exists $\phi \in C^{2}(U)$, $U \Subset \Omega$, touching $u$ from below at
$x_0 \in U$. Then $(-\Delta)^{s}u(x_0)$ is finite and moreover:
\begin{equation}\label{eq}
(-\Delta )^{s}u(x_0) + |\nabla \phi(x_0)|^{p} \geq f(x_0).
\end{equation}A similar result holds for subsolutions.
\end{lemma}

\begin{proof}
We assume that $x_0=0$ and $u(0)=0$. For $r > 0$ so that $B_r :=B(0, r)\subset U$, define:
\begin{equation*}
 \phi_r(x):= \left\lbrace
  \begin{array}{l}
  \phi(x),  \text{ in } B_r\\
    u(x),  \text{ outside  }B_r, \\
  \end{array}
  \right.
\end{equation*}Hence for all $0 < \rho < r$
\begin{equation*}
\begin{split}
\int_{B_r \setminus B_\rho}\frac{u(0)-u(y)}{|y|^{N+2s}}dy&= \int_{B_r \setminus B_\rho}\frac{\phi(y)-u(y)}{|y|^{N+2s}}dy - \int_{B_r \setminus
B_\rho}\frac{\phi(y)}{|y|^{N+2s}}dy \\ & \leq - \int_{B_r \setminus B_\rho}\frac{\phi(y)}{|y|^{N+2s}}dy,
\end{split}
\end{equation*}where we have used that $\phi$ touches $u$ from below. As $\rho \to 0$, the last integral  converges  since $\phi \in \mathcal{C}^{2}(B_r)$. Hence
\begin{equation}\label{limit interval}
\lim_{\rho \to 0}\int_{B_r \setminus B_\rho}\frac{u(0)-u(y)}{|y|^{N+2s}}dy \in [-\infty, M],
\end{equation}where
$$M:= \lim_{\rho \to 0} \left(-\int_{B_r \setminus B_\rho}\frac{\phi(y)}{|y|^{N+2s}}dy\right).$$

Also, from the fact that $u$ is a supersolution, we have $u \geq 0$ in $\mathbb{R}^{N}\setminus \Omega$. Thus
\begin{equation}\label{ineq 1}
\begin{split}
\int_{\mathbb{R}^{N}\setminus B_r }\frac{u(0)-u(y)}{|y|^{N+2s}}dy \leq \int_{\overline{\Omega}\setminus B_r }\frac{-u(y)}{|y|^{N+2s}}dy.
\end{split}
\end{equation}Since $u \in LSC(\overline{\Omega})$, there is a constant $m$ so that
$$u(y) \geq m, \textnormal{ for all }y \in \overline{\Omega}\setminus B_r.$$Hence from \eqref{ineq 1}, it follows
$$\int_{\mathbb{R}^{N}\setminus B_r }\frac{u(0)-u(y)}{|y|^{N+2s}}dy \leq -m \int_{\mathbb{R}^{N}\setminus B_r }\frac{1}{|y|^{N+2s}}dy  < \infty.$$This fact,
together with \eqref{limit interval}, imply that $(-\Delta)^{s}u(0) \in [-\infty, \infty)$.

We now prove the estimate \eqref{eq}, and consequently that $(-\Delta)^{s}u(0)$ is finite. For $\rho> 0$, we have by Lemma \ref{absolute continuity} that
\begin{equation*}
\Big\vert P.V. \int_{B_{r}}\frac{\phi(y)}{|y|^{N+2s}}dy \Big\vert \leq \rho,
\end{equation*}choosing $r$ small enough. Hence
\begin{equation*}
\begin{split}
\int_{\mathbb{R}^{N}\setminus B_{r}}\frac{u(0)-u(y)}{|y|^{N+2s}}dy &=\int_{\mathbb{R}^{N}\setminus B_{r}}\frac{\phi_r(0)-\phi_r(y)}{|y|^{N+2s}}dy
\\&=(-\Delta)^{s}\phi_r(0)-P.V. \int_{B_r}\frac{-\phi(y)}{|y|^{N+2s}}dy \\&\geq -|\nabla \phi(0)|^{p}+f(0)+\rho.
\end{split}
\end{equation*}By letting $r \to 0$, and then $\rho\to 0$, we derive \eqref{eq}.

\end{proof}

We now give the main result of this section.

\begin{theorem}[Comparison principle for viscosity solutions]\label{comparison}
Assume that $f \in \mathcal{C}(\Omega)$. Let $v \in USC(\overline{\Omega})$ be a subsolution  and $u\in LSC(\overline{\Omega})$ be a supersolution, respectively,
of \eqref{boundary problem}. Then $v \leq u$ in $\Omega$.
\end{theorem}
\begin{proof}  We argue by contradiction. Assume that there is $x_0 \in \Omega$ so that:
\begin{equation*}
\sigma := \sup_{\Omega}(v-u) = v(x_0)-u(x_0) >0.
\end{equation*} As usual, we double the variables and consider for $\epsilon > 0$ the function
\begin{equation*}
\Psi_\epsilon(x, y):=v(x)-u(y)-\frac{1}{\epsilon}|x-y|^{2}.
\end{equation*}By the upper semi continuity of $v$ and $-u$, there exist $x_\epsilon$ and $y_\epsilon$ in $\overline{\Omega}$ so that
\begin{equation*}
M_\epsilon:= \sup_{\overline{\Omega}\times \overline{\Omega}}\Psi_\epsilon = \Psi_\epsilon(x_\epsilon, y_\epsilon).
\end{equation*}By compactness, $x_\epsilon \to \overline{x}$ and $y_\epsilon \to \overline{y}$, up to  subsequence that we do not re-label.  From
\begin{equation}\label{ineq with epsilon}
\Psi_\epsilon(x_\epsilon, y_\epsilon) \geq \Psi_\epsilon(x_0, x_0)
\end{equation}and the upper boundedness of $v$ and $-u$ in $\overline{\Omega}$, we derive
$$\lim_{\epsilon \to 0}|x_\epsilon-y_\epsilon|^{2}=0,$$hence $\overline{x}=\overline{y}$. Moreover
$$\Psi_\epsilon(x_\epsilon, y_\epsilon) \geq \Psi_\epsilon(\overline{x}, \overline{x})$$implies that:
$$\lim_{\epsilon \to 0}\frac{1}{\epsilon}|x_\epsilon-y_\epsilon|^{2}=0.$$As a consequence, by letting $\epsilon \to 0$ in \eqref{ineq with epsilon} and using the
semicontinuity of $u$ and $v$, we obtain
\begin{equation}\label{sigma u v}
\sigma = \lim_{\epsilon \to 0}(v(x_\epsilon)-u(y_\epsilon)).
\end{equation}Also, observe that $\overline{x} \in \Omega$, because otherwise there is a contraction with $v \leq u$ in $\mathbb{R}^{N}\setminus\Omega$.

Define the  $\mathcal{C}^2$ test functions
$$\phi_\epsilon(x):= v(x_\epsilon)-\frac{1}{\epsilon}|x_\epsilon-y_\epsilon|^{2}+\frac{1}{\epsilon}|x-y_\epsilon|^{2},$$
$$\psi_\epsilon(y):= u(y_\epsilon)-\frac{1}{\epsilon}|x_\epsilon-y_\epsilon|^{2}+\frac{1}{\epsilon}|x_\epsilon-y|^{2}.$$Then $\phi_\epsilon$ touches $v$ from
above at $x_\epsilon$ and $\psi_\epsilon$ touches $u$ from below at $y_\epsilon$. By Lemma \ref{smooth functions}, we have
\begin{equation*}
(-\Delta)^{s}v(x_\epsilon) + |\nabla \phi_\epsilon(x_\epsilon)|^{p} \leq f(x_\epsilon)
\end{equation*}and
\begin{equation*}
(-\Delta)^{s}u(y_\epsilon) + |\nabla \psi_\epsilon(y_\epsilon)|^{p} \geq f(y_\epsilon)
\end{equation*}
Therefore:
\begin{equation}\label{diff frac lap}
(-\Delta)^{s}v(x_\epsilon)-(-\Delta)^{s}u(y_\epsilon)\leq
f(x_\epsilon)-f(y_\epsilon)+|\nabla\psi_\epsilon(y_\epsilon)|^{p}-|\nabla\phi_\epsilon(x_\epsilon)|^{p}.
\end{equation}Since $f \in \mathcal{C}(\Omega)$ and
$$\nabla_y\psi_\epsilon(y_\epsilon)=-\nabla_x \phi_\epsilon(x_\epsilon),$$we have that the right hand side in \eqref{diff frac lap} tends to $0$ as $\epsilon \to
0$. Thus, we obtain

\begin{equation}\label{ine 2}
\begin{split}
&\liminf_{\epsilon \to 0}\int_{\mathbb{R}^{N}}\frac{v(x_\epsilon)-v(x_\epsilon+z)-u(y_\epsilon)+u(y_\epsilon+z)}{|z|^{N+2s}}dz \\& \qquad \qquad
\qquad=\liminf_{\epsilon \to 0}\left( (-\Delta)^{s}v(x_\epsilon)-(-\Delta)^{s}u(y_\epsilon)\right) \leq 0.
\end{split}
\end{equation}Let $A_{1, \epsilon}:= \{z \in \mathbb{R}^{N}: x_\epsilon + z, y_\epsilon +z \in \Omega\}$. Hence for $z \in A_{1, \epsilon}$, we have from the
inequality
$$\Psi_\epsilon(x_\epsilon, y_\epsilon) \geq \Psi_\epsilon(x_\epsilon+z,y_\epsilon+z )$$that
\begin{equation}\label{non-negative}
v(x_\epsilon)-v(x_\epsilon+z)-u(y_\epsilon)+u(y_\epsilon+z)\geq 0.
\end{equation}Define $A_{2, \epsilon}:= \mathbb{R}^{N}\setminus A_{1, \epsilon}$. We will justify that we are allowed to use Fatou's Theorem in
\begin{equation}\label{need for Fatou}
\liminf_{\epsilon \to 0}\int_{A_{2, \epsilon}}\frac{v(x_\epsilon)-u(y_\epsilon)-v(x_\epsilon+z)+u(y_\epsilon+z)}{|z|^{N+2s}}dz
\end{equation}by showing that the integrand is  bounded from below by an $L^{1}$ function. Firstly, let $r >0$ so that $B_{3r}(\overline{x}) \subset \Omega$ and
take $\epsilon_0$ small enough such that $x_\epsilon, y_\epsilon \in B_r(\overline{x})$ for all $\epsilon < \epsilon_0$.   Take $z \in A_{2, \epsilon}$. We show
now that $|z|\geq 2r$. Indeed, to reach a contradiction, assume that $|z| < 2r$. Since $z \notin A_{1, \epsilon}$, it follows that $x_\epsilon + z$ or $y_\epsilon
+ z$ does not belong to $\Omega$. Without loss of generality, assume $x_\epsilon +z \notin \Omega$. Hence
$$|x_\epsilon +z-\overline{x}|  < 3r,$$and so $x_\epsilon+ z \in B_{3r}(\overline{x}) \subset \Omega $ which is a contradiction.  Next, notice that
\begin{equation}\label{bdd below }
\frac{v(x_\epsilon)-v(x_\epsilon+z)}{|z|^{N+2s}} \geq -\frac{|v(x_\epsilon)|}{|z|^{N+2s}}-\frac{|v_+(x_\epsilon+z)|}{|z|^{N+2s}}.
\end{equation}Hence, using that $z \notin B_{2r}$ when $z \in A_{2, \epsilon}$, we have
$$\int_{A_{2, \epsilon}}\dfrac{|v(x_\epsilon)|}{|z|^{N+2s}} dz \leq C \int_{\mathbb{R}^{N}\setminus B_{2r}}\dfrac{1}{|z|^{N+2s}}dz < \infty.$$On the other hand

\begin{equation*}
\begin{split}
\int_{A_{2, \epsilon}}\dfrac{|v_+(z+x_\epsilon)|}{|z|^{N+2s}}dz &\leq \int_{\mathbb{R}^{N}\setminus B_{2r}}\dfrac{|v_+(z+x_\epsilon)|}{|z|^{N+2s}}dz \\& =
\int_{\mathbb{R}^{N}\setminus B_{2r}(x_\epsilon)}\dfrac{|v_+(y)|}{|y-x_\epsilon|^{N+2s}}dy.
\end{split}
\end{equation*}Since $v$ is a subsolution, we have $v \leq 0$ in $\mathbb{R}^{N}\setminus \Omega$. Hence
\begin{equation*}
\begin{split}
\int_{A_{2, \epsilon}}\dfrac{|v_+(z+x_\epsilon)|}{|z|^{N+2s}}dz & \leq \int_{\overline{\Omega}\setminus
B_{2r}(x_\epsilon)}\dfrac{|v_+(y)|}{|y-x_\epsilon|^{N+2s}}dy \\& \leq \int_{\overline{\Omega}\setminus
B_{r}(\overline{x})}\dfrac{|v_+(y)|}{|y-x_\epsilon|^{N+2s}}dy \\& \leq \frac{1}{r^{N+2s}} \int_{\overline{\Omega}\setminus B_{r}(\overline{x})} v_+(y)dy.
\end{split}
\end{equation*}Observe that the last integral is finite since $v \in L^{1}_{loc}(\mathbb{R}^{N})$ by definition. In this way, recalling \eqref{bdd below }, the
term
$$\frac{v(x_\epsilon)-v(x_\epsilon+z)}{|z|^{N+2s}} $$is bounded from below by an $L^{1}$-integrable function.  A similar result follows for
$$\frac{u(y_\epsilon+z)-u(y_\epsilon)}{|z|^{N+2s}}.$$ Hence, we may use Fatou Lemma in \eqref{need for Fatou} and derive

\begin{equation}\label{long calculation}
\begin{split}
&\liminf_{\epsilon \to 0}\int_{A_{2, \epsilon}}\frac{v(x_\epsilon)-u(y_\epsilon)-v(x_\epsilon+z)+u(y_\epsilon+z)}{|z|^{N+2s}}dz \\ & \qquad\geq\int_{
\mathbb{R}^{N}}\liminf_{\epsilon \to 0}\frac{v(x_\epsilon)-u(y_\epsilon)-v(x_\epsilon+z)+u(y_\epsilon+z)}{|z|^{N+2s}}\chi_{A_{2, \epsilon}}(z) dz \\ & \qquad \geq
\int_{\mathbb{R}^{N}\setminus A_{\overline{x}}}\frac{\sigma +u(\overline{x}+z)-v(\overline{x}+z)}{|z|^{N+2s}}dz.
\end{split}
\end{equation}Here $A_{\overline{x}}:=\{z\in \mathbb{R}^{N}: \overline{x}+z \in \Omega\}$ and we have used  the a. e. pointwise  convergence of $\chi_{A_{2,
\epsilon}}$ to $\chi_{A_{\overline{x}}}$,  \cite[Lemma 4.3]{Brezis} together with a diagonal argument to conclude for a subsequence
$$\liminf_{\epsilon \to 0}[v(x_\epsilon)-u(y_\epsilon)-v(x_\epsilon+z)+u(y_\epsilon+z)] \geq \sigma -v(\overline{x}+z)+ u(\overline{x}+z)$$ for a. e $z \in
\mathbb{R}^{N}$.
Moreover, the inequality $u \geq v$ in $\mathbb{R}^{N}\setminus \Omega$ implies that the last integral in \eqref{long calculation} is non-negative. Then
\begin{equation}\label{nonnegative}
\liminf_{\epsilon \to 0}\int_{A_{2, \epsilon}}\frac{v(x_\epsilon)-u(y_\epsilon)-v(x_\epsilon+z)+u(y_\epsilon+z)}{|z|^{N+2s}}dz \geq 0.
\end{equation}

Therefore by Fatou Lemma, \eqref{nonnegative} and \eqref{ine 2}, we deduce
\begin{equation*}
\begin{split}
&\int_{\mathbb{R}^{N}}\liminf_{\epsilon \to 0}\frac{v(x_\epsilon)-u(y_\epsilon)-v(x_\epsilon+z)+u(y_\epsilon+z)}{|z|^{N+2s}}\chi_{A_{1, \epsilon}}dz\\& \quad
\leq\liminf_{\epsilon \to 0}\int_{A_{1, \epsilon}}\frac{v(x_\epsilon)-u(y_\epsilon)-v(x_\epsilon+z)+u(y_\epsilon+z)}{|z|^{N+2s}}dz  \\ & \quad\leq
\liminf_{\epsilon \to 0}\int_{A_{1, \epsilon}}\frac{v(x_\epsilon)-u(y_\epsilon)-v(x_\epsilon+z)+u(y_\epsilon+z)}{|z|^{N+2s}}dz\\ & \qquad \qquad+
\liminf_{\epsilon \to 0}\int_{A_{2, \epsilon}}\frac{v(x_\epsilon)-u(y_\epsilon)-v(x_\epsilon+z)+u(y_\epsilon+z)}{|z|^{N+2s}}dz\\ & \quad \leq \liminf_{\epsilon
\to 0}\int_{\mathbb{R}^{N}}\frac{v(x_\epsilon)-u(y_\epsilon)-v(x_\epsilon+z)+u(y_\epsilon+z)}{|z|^{N+2s}}dz \leq 0.
\end{split}
\end{equation*}Hence
$$\liminf_{\epsilon \to 0}\frac{v(x_\epsilon)-u(y_\epsilon)-v(x_\epsilon+z)+u(y_\epsilon+z)}{|z|^{N+2s}} \leq 0$$almost everywhere in $A_{1, \epsilon}$. In
particular for $z  \in A_{\overline{x}}$.  We then  have by the lower semicontinuity of $-v$ and $u$  in $\overline{\Omega}$ and \eqref{sigma u v}, that
\begin{equation*}
\begin{split}
0&\geq \liminf_{\epsilon \to 0}\left[ v(x_\epsilon)-u(y_\epsilon)-v(x_\epsilon+z)+u(y_\epsilon+z) \right]\\& \geq \sigma +u(\overline{x}+z)-v(\overline{x}+z).
\end{split}
\end{equation*}Since  $z  \in A_{\overline{x}}$ is arbitrary, we conclude $\sigma \leq v(x)-u(x)$ for a.e. in $\overline{\Omega}$, which implies for $x \in
\partial \Omega$
$$0 \geq v(x)-u(x) \geq \limsup_{y \to x, y \in \Omega}(v(y)-u(y)) \geq \sigma.$$A contradiction with the hypothesis.
\end{proof}
\subsection{Equivalence between strong and viscosity solutions}
In this subsection we prove that strong  and viscosity solutions coincide.
\begin{theorem}\label{first implication}
Any strong  solution $u \in \mathcal{C}^{1, \alpha}(\Omega)$  to problem \eqref{boundary problem} is a viscosity solution as well.
\end{theorem}
\begin{remark}For conditions to ensure the existence of strong  solutions  to problem \eqref{boundary problem} see Theorem \ref{regularity and order int f1},
Theorem \ref{thh} and Corollary \ref{cor classic}.
\end{remark}
\begin{proof}The proof is straightforward, we give it by completeness. Let $u \in \mathcal{C}^{1, \alpha}(\Omega)$ be such that
\begin{equation*}
(-\Delta)^{s}u(x)+ |\nabla u(x)|^{p}=f(x), \quad \text{for all }x \in \Omega.
\end{equation*}Let  $U \subset \Omega$ be open, take $x_0 \in U$ and let $\phi \in \mathcal{C}^{2}(U)$ be such that $u(x_0)=\phi(x_0)$ and $\phi \geq u$ in $U$.
Define
\begin{equation}\label{function v1}
 v(x):= \left\lbrace
  \begin{array}{l}
  \phi(x),  \text{ in } U\\
    u(x),  \text{ outside  }U.\\
  \end{array}
  \right.
\end{equation}Hence, since $u$ is $\mathcal{C}^{1}$, $\nabla u(x_0)= \nabla \phi(x_0)$ and then we have that
$$(-\Delta)^{s}\phi(x_0) + |\nabla \phi(x_0)|^{p} =(-\Delta)^{s}\phi(x_0) + |\nabla u(x_0)|^{p}.$$
By the assumption on $\phi$, we have that $(-\Delta)^{s}\phi(x_0) \leq (-\Delta)^{s}u(x_0)$ and so the $u$ is a viscosity sub-solution. In a similar way, $u$ is a
super-solution and the conclusion follows. \end{proof}
\begin{theorem}\label{second implication}
Assume that the condition \eqref{more} holds that $f \in \mathcal{C}^{\epsilon}(\Omega)\cap L^m(\O)$, for some $\epsilon >0$ and $m>\frac{N}{2s-1}$. We suppose
that $||f||_{L^m(\O)}\le \l^*$ defined in Theorem \ref{thh}. Then any viscosity solution is a strong solution.
\end{theorem}
\begin{proof}
To prove the converse, assume that $u$ is a viscosity solution to problem \eqref{boundary problem}. In view of Theorem \ref{thh} and Corollary \ref{cor classic},
there exists a distributional solution $v$ (which is also strong in view of the assumptions on $f$). Since any  strong solution is of viscosity, we consequently
infer from the Comparison Theorem \ref{comparison} that $u=v$. This ends the proof of the theorem.
\end{proof}

\section{ Some open problems.}

\begin{enumerate}

\item For the existence of solution using approximating argument, the limitation $p<2s$ seems to be technical, we hope that the existence of a solution holds
    for all $p\le 2s$ and for all $f\in L^1(\O)$. For $p>2s$, this is an interesting open question, even  for the Laplacian, with $L^m$ data. Notice that this
    is not the framework of the paper \cite{PPL}.

\item For $p>2s$, it seems to be interesting to eliminate the  smallness condition $||f||_{L^m(\O)}$ and to treat more general set of $p$ without the
    condition \eqref{more}.

\item In order to understand a bigger class of linear integro-differential operators, is seems necessary to obtain alternative techniques independent of the
    representation formula.

\end{enumerate}

\textbf{Acknowledgements.}
\textit{The authors would like to thank the anonymous reviewer for his$/$her careful reading of the paper and his$/$her many insightful comments and suggestions.}


\begin{thebibliography}{00}

\bibitem{Aban} N. Abatangelo, {\it Large S-harmonic functions and boundary blow-up solutions for the fractional Laplacian} , Discrete Contin. Dyn. Syst., 35(12)
    (2015), 5555-5607.

\bibitem{AP} \textsc{B. Abdellaoui, I. Peral}, {\it Towards a deterministic KPZ equation with fractional diffusion: the stationary problem}. Nonlinearity 31,
    2018, 1260-1298. (arXiv:1609.04561v4 [math.AP] 21 Apr 2020.)

\bibitem{AMPP} \textsc{B. Abdellaoui, M. Medina, I. Peral,   A. Primo}, {\it The effect of the Hardy potential in some Calder\'{o}n-Zygmund properties for the fractional Laplacian}. J. Differential Equations, 260 (2016) 8160-8206.

\bibitem{APP} B. Abdellaoui, I. Peral, and A. Primo, {\it Some remarks on non-resonance for elliptic equations with gradient term}.  Ann. Inst. H.
    Poincar\'{e} Anal. Non Lin\'{e}aire 25 (2008), no. 5, 969--985.

\bibitem{AT} O. Alvarez and A. Tourin, \textit{Viscosity solutions of nonlinear integro-differential equations}. Ann. Inst. H. Poincar\'e Anal. Non Lin\'eaire
    \textbf{13} 3 (1996), 293-317.

\bibitem{BP} P. Baras, M. Pierre, \textit{Critere  d'existence  des  solutions  positives  pour  desequations  semi-lineaires  non  monotones.} Ann.  I.H.P. 2 (1985), no.3, 185-212

\bibitem{BI} G. Barles and C. Imbert, \textit{Second-order elliptic integro-differential equations: viscosity solutions' theory revisited}. Ann. Inst. H.
    Poincar\'e Anal. Non Lin\'eaire \textbf{25} 3 (2008),  567-585. \bibitem{BC}C. Bjorland, L. Caffarelli and A. Figalli, \textit{Non-local Tug-of-War and the
    infinite fractional Laplacian}. Adv. Math. \textbf{230} (2012) 1859-1894.

\bibitem{BGO} \textsc{L. Boccardo, T. Gallou\"{e}t, L. Orsina}, {\it Existence and nonexistence of solutions for some nonlinear elliptic equations}, J. Anal.
    Math. 73 (1997) 203-223.


\bibitem{BJ} \textsc{K. Bogdan, T. Jakubowski}, {\it Estimates of the Green Function for the Fractional Laplacian Perturbed by Gradient}, Potential Anal  36
    (2012), 455-481.

\bibitem{Brezis} H. Brezis, \textit{Functional analysis, Sobolev spaces and partial differential equations}. Springer, 2011.

\bibitem{BV} C. Bucur and  E. Valdinocci, \textit{Nonlocal diffusion and applications}. Springer, 2016.



\bibitem{C} L. Caffarelli, \textit{Non local operators, drifts and games}. Nonlinear PDEs. Abel Symposia \textbf{7} (2012), 37-52.

\bibitem{CVa} L. Caffarelli and A. Vasseur, \textit{Drift diffusion equations with fractional diffusion and the quasi- geostrophic equation}. Ann. Math. (2010)
    1903-1930. \bibitem{CS} L. Caffarelli and L. Silvestre, \textit{Regularity theory for fully-nonlinear integro-differential equations}. Comm, Pure Appl. Math.
    \textbf{62} (2009), 597-638. \bibitem{ChenS} Z. Chen and R. Song, \textit{Estimates for Green functions and Poisson kernels for symmetric stable processes}.
    Mathematische Annalen \textbf{312} (1998), 465-501.

\bibitem{CV} H. Chen and L. Veron, \textit{Semilinear fractional elliptic equations with gradient nonlinearity involving measures}. Journal of Functional
    Analysis. \textbf{266} 8 (2014), 5467-5492.

\bibitem{CV1} H. Chen and L. Veron, \textit{Semilinear fractional elliptic equations involving measures}. Journal of Differential Equations.  \textbf{257} 5
    (2014), 1457-1486.

\bibitem{CIL} M. Crandall, H. Ishii and P-L. Lions, \textit{User's guide to viscosity solutions of second order partial differential equations}.  Bull. of Amer.
    Soc. \textbf{27}  1 (1992), 1-67.

\bibitem{dine} {\sc E. Di Nezza, G. Palatucci,  E. Valdinoci}, {\em Hitchhiker's guide to the fractional Sobolev
    spaces}, Bull. Sci. Math. {\bf 136} (2012), no. 5, 521-573.


\bibitem{FL} I. Fonseca and G. Leoni, \textit{Modern methods in the Calculus of Variations: $L^{p}$ spaces}. Springer, 2007.

\bibitem{G} \newblock Y. Giga,  (2006).  \newblock \emph{Surface evolution equations: a level set method}.  \newblock Basel: Monographs in Mathematics 99,
    Birkhauser Verlag.


\bibitem{GO} G. Gilboa and S. Osher, \textit{Non-local operators with applications to image processing.} Miltiscale Model. Simul. \textbf{7} (2008), 1005-1028.

\bibitem{HMV}\textsc{ K Hansson, V.G. Maz'ya, I.E. Verbitsky}, {\it Criteria of solvability for multidimensional Riccati equations}, Ark. Mat., {\bf 37}, (1999),
    87-120.

\bibitem{I} Ishii, H., Perron's method for Hamilton-Jacobi equations. \textit{Duke math. J}. \textbf{55} 2 (1987), 369-384.

\bibitem{KPZ} M. Kardar, G. Parisi and Y.C. Zhang, \textit{Dynamic scaling of growing interfaces}. Phys. Rev. Lett. \textbf{56} (1986), 889-892.

\bibitem{KKL} J. Korvenp\"{a}\"{a}, T. Kuusi and E. Lindgren, \textit{Equivalence of solutions to fractional p-Laplace type equations}. To appear.

\bibitem{Li} E. Lindgren, \textit{H\" older estimates for viscosity solutions of equations of fractional $ p $-Laplace type}. arXiv preprint arXiv:1405.6612
    (2014).

\bibitem{PPL} P.L. Lions \textit{ R\'{e}solution de probl\`{e}mes quasilin\'{e}aires}, s
    Archive for Rational Mechanics and Analysis 74 (4),(1980), 335-353


\bibitem{MV} \textsc{V.G. Maz'ya, I.E. Verbitsky}, {\it Capacitary estimates for fractional integrals, with applications to partial differential equations and
    Sobolev multipliers,} Ark. Mat. 33 (1995), 81-115.

\bibitem{MK} R. Metzler and  J. Klafter, \textit{The random walk's guide to anomalous diffusion: a fractional dynamics approach}. Phys. Rep. \textbf{339} (2000),
    1-77.

\bibitem{MK1} R. Metzler and J. Klafter,  \textit{The restaurant at the random walk: recent developments in the description of anomalous transport by fractional
    dynamics}. J. Phys. A \textbf{37} (2004), 161-208.

\bibitem{RS} X. Ros-Oton and J. Serra, \textit{The Dirichlet problem for the fractional Laplacian: regularity up to the boundary}. Journal de Math\'{e}matiques
    Pures et Appliqu\'{e}es  \textbf{101} 3 (2014), 275-302.

\bibitem{SV}R. Servadei and E. Valdinoci, \textit{Weak and viscosity solutions of the fractional Laplace equation}. Publ. Mat. \textbf{58} 1 (2014), 133-154.

\bibitem{SW} R. Showalter, \textit{Monotone operators in Banach spaces and partial differential equations}. AMS, 1997.

\bibitem{sil}{\sc L. Silvestre}, {\em Regularity of the obstacle problem for a fractional power of the Laplace operator}, Comm. Pure Appl. Math. 60 (2007), no. 1, 67-112.

\bibitem{S} L. Silvestre, \textit{H\"{o}lder estimates for solutions of integral-differential equations like the fractional Laplacian}. Indiana Univ. Math. J.
    \textbf{55} 3  (2006), 1155-1174.

\bibitem{Sil} L. Silvestre, \textit{On the differentiability of the solution to the Hamilton-Jacobi equation with critical fractional diffusion}, Advances in
    Mathematics 226 (2011), 2020-2029.

\bibitem{War} M. Warma, \textit{The Fractional Relative Capacity and the Fractional Laplacian with Neumann
and Robin Boundary Conditions on Open Sets.} J. Potential Analysis, 42, (2015), no 2, 499-547.
\end{thebibliography}
\end{document}